\documentclass{article}
\usepackage{graphicx,amsfonts,amsmath,amssymb}
\begin{document}

\newtheorem{theorem}{{\bf Theorem}}[section]
\newtheorem{fact}[theorem]{{\bf Fact}}
\newtheorem{lemma}[theorem]{{\bf Lemma}}
\newtheorem{sublemma}[theorem]{{\bf Sublemma}}
\newtheorem{proposition}[theorem]{{\bf Proposition}}
\newtheorem{corollary}[theorem]{{\bf Corollary}}
\newtheorem{definition}[theorem]{{\bf Definition}}
\newtheorem{notation}[theorem]{{\bf Notation}}
\newtheorem{convention}{{\bf Convention}}
\newtheorem{case}{{\bf Case}}
\newtheorem{subcase}{{\bf Case 1.\hspace{-0.2em}}}
\newtheorem{subcase2}{{\bf Case 2.\hspace{-0.2em}}}
\newtheorem{problem}{{\bf Problem}}
\newenvironment{proof}{{\bf Proof }}{\hfill$\square$}
\newtheorem{remark}[theorem]{{\bf Remark}}
\newtheorem{claim}{{\bf Claim}}
\newcommand\Z{{\mathbb Z}}
\newcommand\N{{\mathbb N}}
\newcommand\Q{{\mathbb Q}}
\newcommand\R{{\mathbb R}}
\newcommand\diff{{\rm Diff}}
\newcommand\sym{{\rm Sym}}
\newcommand\fix{{\rm Fix}}
\newcommand\out{{\rm Out}}
\newcommand\aut{{\rm Aut}}
\newcommand\inn{{\rm Inn}}
\newcommand\id{{\rm id}}
\newcommand\K{{\mathcal K}}
\newcommand\T{{\mathcal T}}
\newcommand\LL{{\mathcal L}}
\newcommand\D{{\mathcal D}}
\newcommand\mcg{{\mathcal M}}
\newcommand\pmcg{{\mathcal PM}}
\newcommand\isom{{\rm Isom}}
\newcommand\interior{{\rm Int}}
\newcommand\cl{{\rm cl}}
\relpenalty=10000
\binoppenalty=10000
\uchyph=-1

\makeatletter
\def\tbcaption{\def\@captype{table}\caption}
\def\figcaption{\def\@captype{figure}\caption}
\makeatother

\title{Characterization of 3-bridge links with infinitely many 3-bridge spheres} 
\author{Yeonhee Jang
\footnote{This research is partially supported by Grant-in-Aid for JSPS Research Fellowships for Young Scientists.}\\
Department of Mathematics, Hiroshima University, \\
Hiroshima 739-8526, Japan\\
yeonheejang@hiroshima-u.ac.jp}

\date{\empty}
\maketitle

\begin{abstract}
In \cite{Jan}, the author constructed an infinite family of 3-bridge links 
each of which admits infinitely many 3-bridge spheres up to isotopy.
In this paper, we prove that if a prime, unsplittable link $L$ in $S^3$ admits infinitely many 3-bridge spheres up to isotopy
then $L$ belongs to the family.
\end{abstract}






\section{Introduction}\label{intro}

An {\it $n$-bridge sphere} of a link $L$ in $S^3$
is a 2-sphere which meets $L$ in $2n$ points
and cuts $(S^3, L)$ into $n$-string trivial tangles
$(B_1, t_1)$ and $(B_2, t_2)$. 
Here, an {\it $n$-string trivial tangle} is 
a pair $(B^3, t)$ of the $3$-ball $B^3$ and 
$n$ arcs properly embedded in $B^3$ parallel to the boundary of $B^3$. 
We call a link $L$ an {\it $n$-bridge link} 
if $L$ admits an $n$-bridge sphere and does not admit an ($n-1$)-bridge sphere.
Two $n$-bridge spheres $S_1$ and $S_2$ of $L$ are said to be {\it pairwise isotopic} 
({\it isotopic}, in brief)
if there exists a homeomorphism $f:(S^3, L)\rightarrow (S^3, L)$
such that $f(S_1)=S_2$ and 
$f$ is {\it pairwise isotopic} to the identity,
i.e., there is a continuous family of homeomorphisms 
$f_t:(S^3, L)\rightarrow (S^3, L)$ $(0\leq t\leq 1)$
such that $f_0=f$ and $f_1=\id$.

It is known
by Otal (\cite{Ota1} and \cite{Ota2})
that the unknot (resp. any 2-bridge link)
admits a unique $n$-bridge sphere up to isotopy 
for $n\geq 1$ (resp. $n\geq 2$).
These results were recently refined 
by Scharlemann and Tomova \cite{Sch2}.
The author constructed an infinite family of links 
each of which admits infinitely many 3-bridge spheres up to isotopy in \cite{Jan},
and gave a classification of 3-bridge spheres of 3-bridge arborescent links in \cite{Jan3}.
In this paper, we prove the following theorem.

\begin{theorem}\label{thm-main}
Let $L$ be a prime, unsplittable link in $S^3$.
Then $L$ admits infinitely many 3-bridge spheres up to isotopy
if and only if 
$L$ is equivalent to a link $L(q/2p;\beta_1/\alpha_1,\beta_2/\alpha_2)$ (see Figure \ref{fig-3b-inf}) with $q\not\equiv 1\pmod{p}$ and $|\alpha_1|>1$ (or $|\alpha_2|>1$).
\end{theorem}
%
\begin{figure}
\begin{center}
\includegraphics*[width=9.5cm]{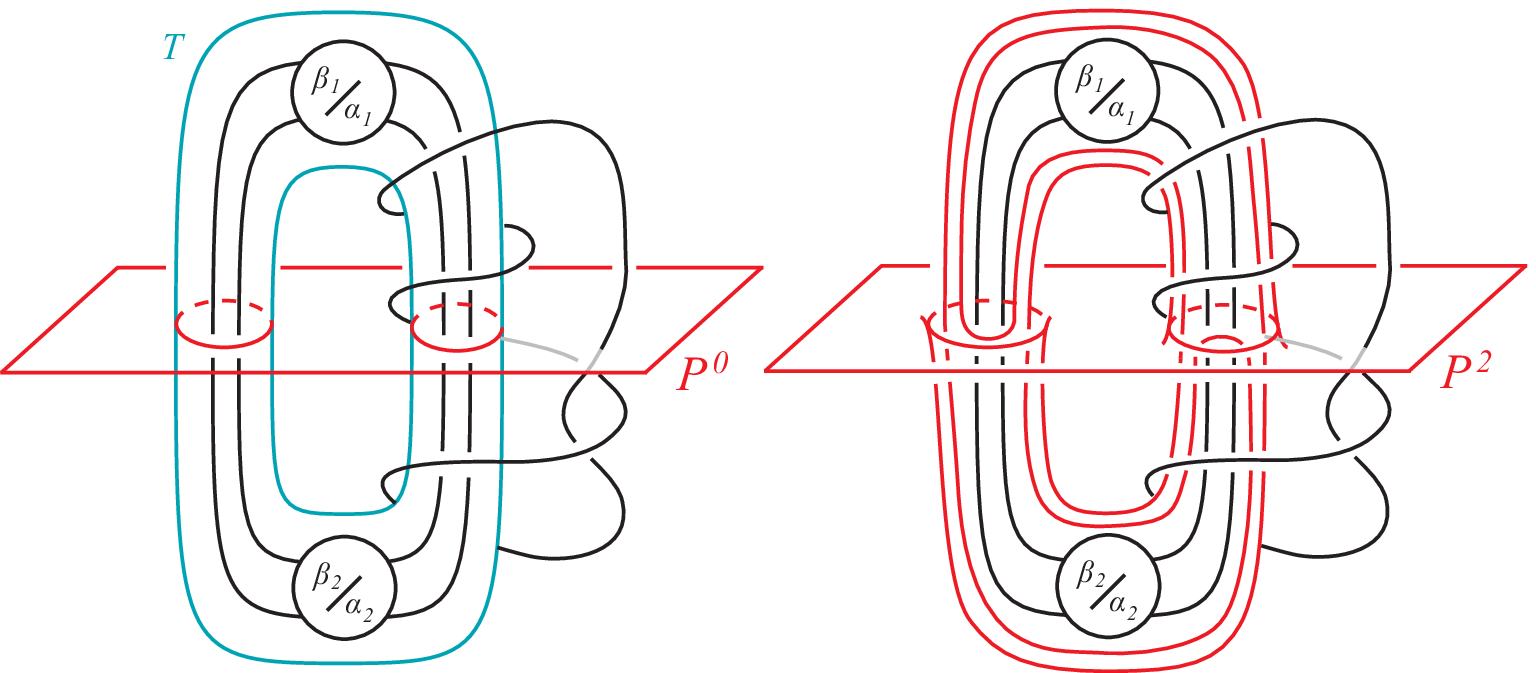}
\end{center}
\caption{}
\label{fig-3b-inf}
\end{figure}
%

Here, two links are said to be {\it equivalent} 
if there exists an orientation-preserving homeomorphism of $S^3$ which sends one to the other.
The link $L(q/2p;\beta_1/\alpha_1,\beta_2/\alpha_2)$ in Figure \ref{fig-3b-inf}
is obtained as follows.
Let $V_0$ be a solid torus standardly embedded in $S^3$ and $L_0$ a link in $V_0$ obtained by connecting two rational tangles
of slopes $\beta_1/\alpha_1$ and $\beta_2/\alpha_2$ by \lq\lq trivial arcs\rq\rq\  as in the figure.
Let $K_1\cup K_2$ be a 2-bridge link in $S^3$ of type $(2p,q)$ and $V_1$ the regular neighborhood of $K_1$.
For $i=0,1$, let $l_i$ be the preferred longitude of $V_i$, that is, $l_i$ is an essential loop on $\partial V_i$ which is null-homologous in $S^3\setminus \interior(V_i)$.
Let $h:V_0\rightarrow V_1$ be a homeomorphism which carries $l_0$ to $l_1$.
We denote by $L(q/2p;\beta_1/\alpha_1,\beta_2/\alpha_2)$ the union of $h(L_0)$ and $K_2$.

\begin{remark}\label{rmk-main}
{\rm
We can also see that any 3-bridge sphere of $L(q/2p;\beta_1/\alpha_1,\beta_2/\alpha_2)$
is isotopic to $P^i$ for some integer $i$,
where $P^i$ is obtained from $P^0$ by applying the $i$-th power of the \lq\lq half Dehn twist\rq\rq\  along the torus $T$ as illustrated in Figure \ref{fig-3b-inf} (see \cite{Jan} for detailed description of $P^i$).
This implies that any prime, unsplittable link admits only finitely many 3-bridge spheres up to homeomorphism.
}
\end{remark}

Theorem \ref{thm-main} gives a partial answer to an analogy of the Waldhausen conjecture 
in terms of knot theory,
namely, a prime, unsplittable link with atoroidal complement admits only finitely many $n$-bridge spheres up to isotopy for a given $n(\in\N)$.
The (original) Waldhausen conjecture asserts that 
a closed orientable atoroidal 3-manifold admits only finitely many Heegaard splittings of given genus $g(\in \N)$ up to isotopy
and was proved to be true by Johannson \cite{Joh2} and Li \cite{Li}.


\section{Heegaard splittings of 3-manifolds}\label{sec-hs}

Let $M$ be a closed orientable 3-manifold.
A genus-$g$ {\it Heegaard splitting} of $M$ is a tuple $(V_1, V_2; F)$,
where $V_1$ and $V_2$ are genus-$g$ handlebodies in $M$
such that $M=V_1\cup V_2$ and $F=\partial V_1=\partial V_2=V_1\cap V_2$.
Two Heegaard splittings $(V_1, V_2; F)$ and $(W_1, W_2; G)$ of a 3-manifold $M$ 
are said to be {\it isotopic}
if there exists a self-homeomorphism $f$ of $M$ 
such that $f(F)=G$ and $f$ is isotopic to the identity map $\id_M$ on $M$.

For a genus-$2$ Heegaard splitting $(V_1, V_2; F)$ of $M$,
it is known that  
there is an involution $\tau_F$ on $M$ 
satisfying the following condition.

\begin{itemize}
\item[($\ast$)] $\tau_F(V_i)=V_i$ $(i=1,2)$
and $\tau_F|_{V_i}$ is {\it equivalent} to the standard involution $\T$
on a standard genus-2 handlebody $V$ as illustrated in Figure \ref{fig-hs-3b}.
To be precise, there is a homeomorphism $\psi_i:V_i\rightarrow V$
such that $\T=\psi_i(\tau_F|_{V_i})\psi_i^{-1}$ $(i=1,2)$.
\end{itemize}
%
\begin{figure}
\begin{center}
\includegraphics*[width=8cm]{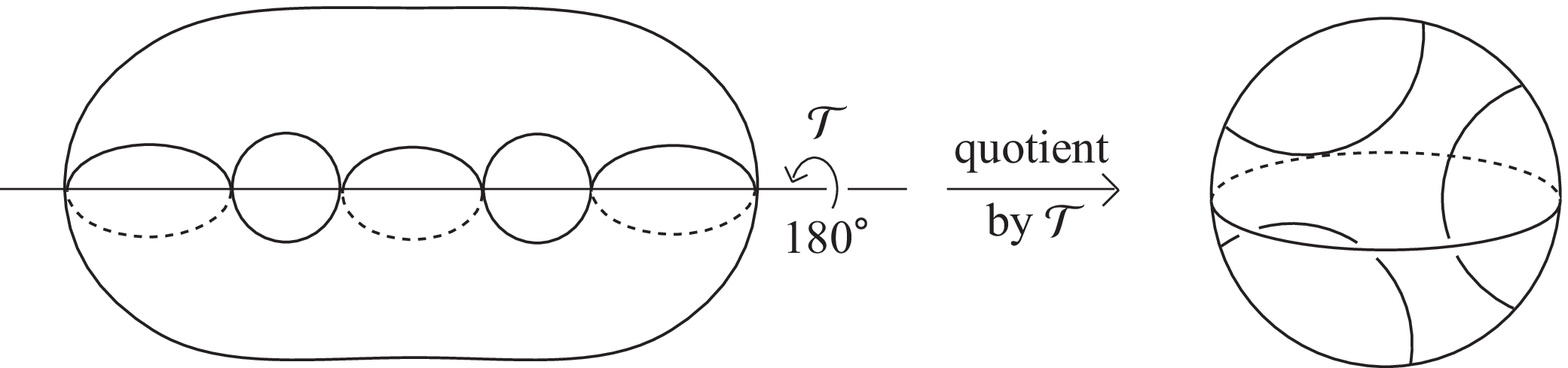}
\end{center}
\caption{}
\label{fig-hs-3b}
\end{figure}
%

The strong equivalence class of $\tau_F$
is uniquely determined by the isotopy class of $(V_1, V_2; F)$ (cf. \cite[Proposition 5]{Jan2})
and we call $\tau_F$ the {\it hyper-elliptic involution} 
associated with $(V_1, V_2; F)$ 
(or associated with $F$, in brief).
Here, 
two involutions $\tau$ and $\tau'$ are said to be {\it strongly equivalent}
if there exists a homeomorphism $h$ on $M$ 
such that $h\tau h^{-1}=\tau'$ and that 
$h$ is isotopic to the identity map $\id_M$.

Let $L$ be a prime, unsplittable 3-bridge link.
Let $M$ be the double branched covering of $S^3$ branched along $L$
and $\tau_L$ the covering transformation.
Let $\Phi_L$ be the natural map 
from the set of isotopy classes of 3-bridge spheres of $L$
to the set of isotopy classes of genus-2 Heegaard surfaces of $M$ whose hyper-elliptic involution is $\tau_L$. 
The following proposition is proved in \cite{Jan3}.

\begin{proposition}\label{prop-hs-3b}
$\Phi_L$ is at most $2$-$1$. 
Moreover, $\Phi_L$ is injective if $L$ is not a non-elliptic Montesinos link.
\end{proposition}

In the rest of this section,
we recall a characterization of genus-$2$ $3$-manifolds admitting nontrivial torus decompositions due to Kobayashi \cite{Kob} (and \cite{Kob2}).
We use the following notation.

\begin{itemize}
\item[{$D[r]$}] (resp. $M\ddot{o}[r], A[r]$) : 
the set of all orientable Seifert fibered spaces 
over a disk $D$ (resp. a M\"{o}bius band $M\ddot{o}$, an annulus $A$) 
with $r$ exceptional fibers.

\item[$M_K$] :
the set of the exteriors of the nontrivial 2-bridge knots.

\item[$M_L$] :
the set of the exteriors of the nontrivial 2-component 2-bridge links.

\item[$L_K$]: the set of the exteriors of the 1-bridge knots in lens spaces 
each of which admits a complete hyperbolic structure or admits a Seifert fibration
whose regular fiber is not a meridian loop.

\item[$KI$]: the twisted $I$-bundle on the Klein bottle.

\end{itemize}

If $K$ is a 2-bridge knot (resp. a 2-bridge link, a 1-bridge knot in a lens space),
$E(K)$ denotes a manifold in $M_K$ (resp. $M_L$, $L_K$)
obtained as the exterior of $K$.

\begin{theorem}\label{thm-kob}
Let $M$ be a closed, connected, orientable Haken $3$-manifold of Heegaard genus $2$
which admits a nontrivial torus decomposition.
Let $(V_1, V_2; F)$ be a genus-$2$ Heegaard splitting of $M$.
Then $M$ satisfies one of the following four conditions (1), (2), (3) and (4), and
$F$ is isotopic to a Heegaard surface, denoted by the same symbol $F$,
as follows 
(see Figure \ref{fig-hs}).
%
\begin{figure}
\begin{center}
\includegraphics*[width=12.8cm]{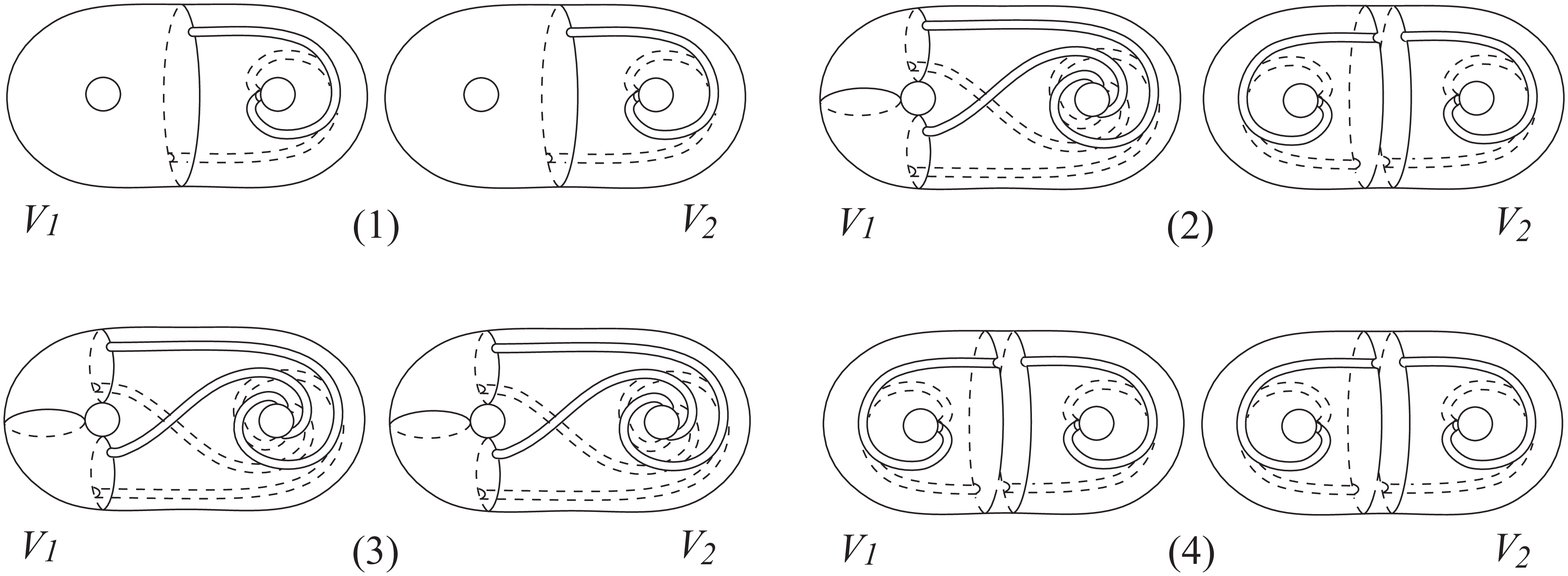}
\end{center}
\caption{}
\label{fig-hs}
\end{figure}

\begin{enumerate}
\item[{\rm (1)}] 
$M$ is obtained from $M_1\in D[2]$ and $M_2=E(K)\in L_K$ 
by identifying their boundaries so that 
the regular fiber of $M_1$ is identified with the meridian loop of $K$.
Moreover, 
\begin{itemize}
\item $M_1\cap F$ is an essential annulus saturated in the Seifert fibration of $M_1$, and
\item $M_2\cap F$ is a 2-holed torus which gives a 1-bridge decomposition of the 1-bridge knot $K$.
\end{itemize}
Moreover, $V_i\cap T$ $(i=1,2)$ consists of a single separating essential annulus,
where $T=\partial M_1=\partial M_2$.

\item[{\rm (2)}] 
$M$ is obtained from $M_1\in D[2]\cup D[3]$ and $M_2=E(K)\in M_K$ 
by identifying their boundaries so that 
the regular fiber of $M_1$ is identified with the meridian loop of $K$.
Moreover, 
\begin{itemize}
\item $M_1\cap F$ consists of two disjoint essential saturated annuli in $M_1$ 
which divide $M_1$ into three solid tori, and
\item $M_2\cap F$ is a 2-bridge sphere of the nontrivial 2-bridge knot $K$.
\end{itemize}

Moreover, by exchanging $V_1$ and $V_2$ if necessary, 

\begin{itemize}
\item[{\rm (i)}]
$V_1\cap T$ consists of 
two disjoint non-separating essential annuli 
satisfying the following condition: 
there exists a complete meridian disk system $(D_1, D_2)$ of $V_1$ 
such that $D_1\cap (V_1\cap T)=\emptyset$
and  $D_2\cap (V_1\cap T)$ 
consists of essential arcs properly embedded in each annulus of $V_1\cap T$, and
\item[{\rm (ii)}]
$V_2\cap T$ consists of 
disjoint non-parallel separating essential annuli,
\end{itemize}
where $T=\partial M_1=\partial M_2$.

\item[{\rm (3)}] 
$M$ is obtained from 
\begin{itemize}
\item[{\rm (3-1)}] $M_1\in M\ddot{o}[r]$ $(r=0,1,2)$ and $M_2=E(K)\in M_K$, or
\item[{\rm (3-2)}] $M_1\in A[r]$ $(r=0,1,2)$ and $M_2=E(K)\in M_L$
\end{itemize}
by identifying their boundaries so that 
the regular fiber of $M_1$ is identified with the meridian loop of $K$.
Moreover, 
\begin{itemize}
\item $M_1\cap F$ consists of two disjoint essential saturated annuli in $M_1$ 
which divide $M_1$ into two solid tori, and
\item $M_2\cap F$ is a 2-bridge sphere of the 2-bridge link $K$.
\end{itemize}
Moreover, $V_i\cap T$ $(i=1,2)$ consists of two disjoint non-separating essential annuli 
satisfying the condition {\rm (i)} of (2),
where $T=\partial M_1=\partial M_2$.

\item[{\rm (4)}] 
$M$ is obtained from $M_1, M_2\in D[2]$ and 
$M_3=E(K_1\cup K_2)\in M_L$
by identifying their boundaries so that 
the regular fiber of $M_i$ is identified with the meridian loop of $K_i$ ($i=1,2$).
Moreover, 
\begin{itemize}
\item $M_i\cap F$ is an essential saturated annulus in $M_i$ $(i=1,2)$, and
\item $M_3\cap F$ is a 2-bridge sphere of the 2-bridge link $K_1\cup K_2$.
\end{itemize}
Moreover, $V_i\cap T$ $(i=1,2)$ consists of 
two disjoint non-parallel separating essential annuli 
satisfying the condition {\rm (ii)} of (2),
where $T=\partial M_1\cup \partial M_2=\partial M_3$.

\end{enumerate}
\end{theorem}

\begin{proof}
Let $\Gamma$ be the union of tori which gives the torus decomposition of $M$.
If each component of $\Gamma$ is separating,
then we see from the proof of the main theorem of \cite{Kob}
that $M$ and $F$ satisfies one of the conditions (1), (2), (3-1) and (4), where $\Gamma=T$.
In the rest of this proof, we assume that $\Gamma$ has a non-separating component 
and show that $M$ and $F$ satisfy the condition (3-2).
By the proof of the main theorem of \cite{Kob}, 
$M$ satisfies the condition (3-2).
In particular, $M$ is obtained by gluing $M_1\in A[r]$ $(r=0, 1, 2)$ and 
$M_2=E(K_1\cup K_2)$, where $K_1\cup K_2$ is a nontrivial 2-bridge link.
In the rest of this proof, we see that $F$ satisfies the condition (3-2) in this case.

First assume that neither $M_1$ nor $M_2$ is homeomorphic to $S^1\times S^1\times I$.
Then $\Gamma$ consists of two components.
By an argument similar to that for the main theorem of \cite{Kob}, 
together with Lemmas 3.1, 3.2 and 3.3 of \cite{Kob2},
we can see that $F$ satisfies the condition (3-2).

In the remainder of this proof, 
assume that either $M_1$ or $M_2$ is homeomorphic to $S^1\times S^1\times I$.
Then we may assume that $\Gamma$ is a component of $T$.

If $M_1$ is homeomorphic to $S^1\times S^1\times I$,
then $M\setminus \Gamma$ is homeomorphic to the interior of $M_2$.
By an argument similar to that for the main theorem of \cite{Kob}, 
together with Lemmas 3.1, 3.2 and 3.3 of \cite{Kob2},
we can see that $\Gamma\cap V_i$ is a non-separating annulus as illustrated in Figure \ref{fig-hs-r0}.
%
\begin{figure}
\begin{center}
\includegraphics*[width=3.9cm]{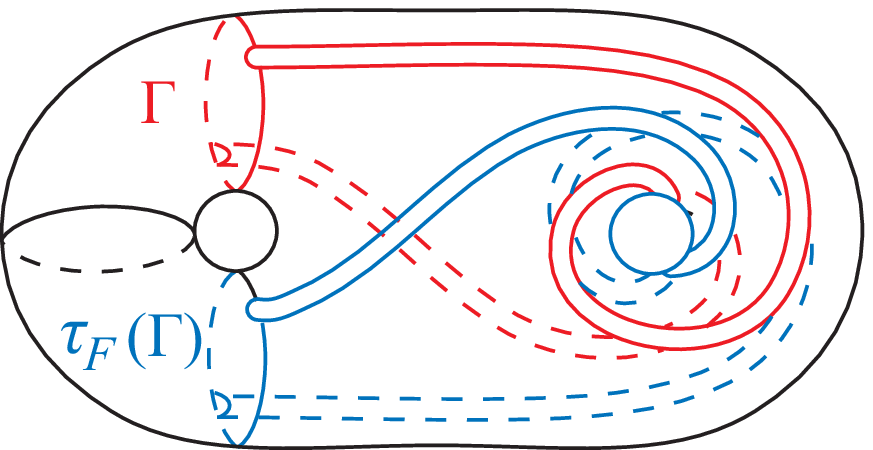}
\end{center}
\caption{}
\label{fig-hs-r0}
\end{figure}
%
Let $\tau_F$ be the hyper-elliptic involution associated with $F$.
Then $\tau_F(\Gamma)$ is an essential torus in $M\setminus \Gamma$ (see Figure \ref{fig-hs-r0}).
Note that $M_2$ is homeomorphic to $A(1/n)$ for some integer $n$ or is hyperbolic
(see \cite[Lemma 4.4]{Kob} and see \cite[Section 2]{Jan2} for notation).
Thus any essential torus in $M_2$ is $\partial$-parallel,
and hence, $\tau_F(\Gamma)$ is isotopic to $\Gamma$ in $M$.
This implies that $T$ is isotopic to $\Gamma\cup\tau_F(\Gamma)$ 
since $M_1$ is homeomorphic to $S^1\times S^1\times I$, 
and we see that $F$ satisfies the condition (3).

If $M_2=E(K_1\cup K_2)$ is homeomorphic to $S^1\times S^1\times I$,
then $K_1\cup K_2$ is a Hopf link and
$M\setminus \Gamma$ is homeomorphic to the interior of $M_1$.
By an argument similar to that in the previous case,
we see that $\Gamma\cap V_i$ is a non-separating annulus as illustrated in Figure \ref{fig-hs-r0}
and that $\tau_F(\Gamma)$ is also a non-separating essential torus in $M\setminus \Gamma$.
If $M\setminus \Gamma(\cong M_1)\in A[r]$ for $r\leq 1$, then any essential torus in $M_1$ is $\partial$-parallel,
and hence, we see that $F$ satisfies the condition (3) by an argument similar to that in the previous case.
(We use the same symbol $M\setminus \Gamma$ to denote the manifold
obtained by closing-up $M\setminus \Gamma$ with two tori.)
If $M\setminus \Gamma(\cong M_1)\in A[2]$, 
then any essential torus in $M_1$ is either $\partial$-parallel 
or an essential torus which divides $M_1$ into two Seifert fibered spaces belonging to $A[1]$.
Hence, 
$\Gamma\cup\tau_F(\Gamma)$ cuts $M$ into two Seifert fibered spaces which belong to $A[1]$
whose fibrations are identical on $\tau_F(\Gamma)$.
On the other hand, by an argument in \cite[Section 6]{Kob} (see also Figure \ref{fig-hs-r0}),
one of the two Seifert fibered spaces must be the exterior of a 2-bridge link, say $L'$, and
the meridians of $L'$ must be identified with regular fibers of the other Seifert fibered space.
Since the fibrations of the two Seifert fibered space are identical on $\tau_F(\Gamma)$,
this implies that the meridian of a component of $L'$ is a regular fiber of $E(L')(\in A[1])$.
However, this is impossible (cf. \cite[Lemma 1]{Jan2} or \cite[Lemma 4.4]{Kob}).
Hence, any essential torus in $M_1$ is $\partial$-parallel,
and we see that $F$ satisfies the condition (3)
\end{proof}
\vspace{2mm}

We need to study the manifolds satisfying one of the following conditions
in the proof of Theorem \ref{thm-main} (cf. Section \ref{sec-proof}).

\begin{itemize}
\item[(M1)] $M$ is obtained by gluing $M_1\in D[2]$ and $M_2=L(p,q)\setminus N(K)$ as in Theorem \ref{thm-kob} (1), where $M_2$ is hyperbolic,
\item[(M2)] $M$ is obtained by gluing $M_1\in D[2]\cup D[3]$ and $M_2=E(K)\in M_K$ as in Theorem \ref{thm-kob} (2), where $M_2$ is hyperbolic,
\item[(M3-1-1)] $M$ is obtained by gluing $M_1\in M\ddot{o}[r]$ $(r=1,2)$ and $M_2=E(K)\in M_K$ as in Theorem \ref{thm-kob} (3),
where $K$ is a torus knot of type $(2,n)$,
\item[(M3-1-2)] $M$ is obtained by gluing $M_1\in M\ddot{o}[r]$ $(r=0,1,2)$ and $M_2=E(K)\in M_K$ as in Theorem \ref{thm-kob}(3), where $M_2$ is hyperbolic,
\item[(M3-2-1)] $M$ is obtained by gluing $M_1\in A[r]$ $(r=0,1,2)$ and $M_2=E(K)\in M_L$ as in Theorem \ref{thm-kob} (3),
where $K$ is a torus link of type $(2,n)$,
\item[(M3-2-2)] $M$ is obtained by gluing $M_1\in A[r]$ $(r=0,1,2)$ and $M_2=E(K)\in M_L$ as in Theorem \ref{thm-kob}(3), where $M_2$ is hyperbolic,
\item[(M4)] $M$ is obtained by gluing $M_1, M_2\in D[2]$ and $M_3=E(K)\in M_K$ as in Theorem \ref{thm-kob} (4), where $M_3$ is hyperbolic.
\end{itemize}

\begin{remark}\label{rmk-dbc}
{\rm 
The double branched covering of $S^3$ branched over $L(q/2p;\beta_1/\alpha_1, \beta_2/\alpha_2)$
satisfies the condition (M3-1-2) or (M3-2-2), where $r\geq 1$ (cf. \cite{Jan}).
}
\end{remark}


\section{Mapping class groups}\label{sec-mcg}

In this section,
we calculate certain subgroups of the mapping class groups of 
the Seifert fibered spaces and the manifolds which arose in Theorem \ref{thm-kob}.
This enables us to compare the hyper-elliptic involutions of genus-2 Heegaard surfaces of 3-manifolds.
For a hyperbolic 3-manifold $N$,
let $\mcg(N)$ be the (orientation-preserving) mapping class group of $N$.
For a Seifert fibered space $N$,
let $\mcg(N)$ be the subgroup of the (orientation-preserving) mapping class group of $N$ 
which consists of elements preserving each singular fiber of $N$.
(See \cite{Jan2} for more details.)
When $N$ is a Seifert fibered space over a surface $F$, 
let $\mcg^0(N)$ be the subgroup of $\mcg(N)$ which consists of the elements inducing the identity map on $F$, 
$\mcg^{\ast}(F)$ the mapping class group of $(F, {\rm exceptional\ points})$.
For a 3-manifold $M$ in Theorem \ref{thm-kob},
let $\mcg(M)$ be the subgroup of the (orientation-preserving) mapping class group of $M$ 
which consists of the elements preserving each piece of the torus decomposition of $M$ and each singular fiber of the Seifert pieces.

We describe some elements of the mapping class groups of certain Seifert fibered spaces.
Let $N_1$ and $N_2$ be the Seifert fibered spaces 
$M\ddot{o}(\beta_1/\alpha_1,\beta_2/\alpha_2)\in M\ddot{o}[2]$ and $A(\beta_1/\alpha_1,\beta_2/\alpha_2)$\\$\in A[2]$, respectively
(see \cite[Section 2]{Jan2} for notation). 
We define $g_i, b, D_j\in \mcg(N_1)$ and $h_i, a, D_j'\in \mcg(N_2)$ $(i,j\in\{1,2\})$ as follows.
We denote by $g_i$ and $h_i$ the involutions as illustrated in Figure \ref{fig-mcg}.
The symbols $a$ and $b$ denote the Dehn twist along saturated annuli $A_a$ and $A_b$, respectively, in the direction of a fiber, and
$D_j$ and $D_j'$ are the Dehn twists along saturated tori $T_{D_j}$ and $T_{D_j'}$, respectively, in the direction of loops intersecting regular fibers in one point,
as illustrated in Figure \ref{fig-mcg}.
%
\begin{figure}
\begin{center}
\includegraphics*[width=11.5cm]{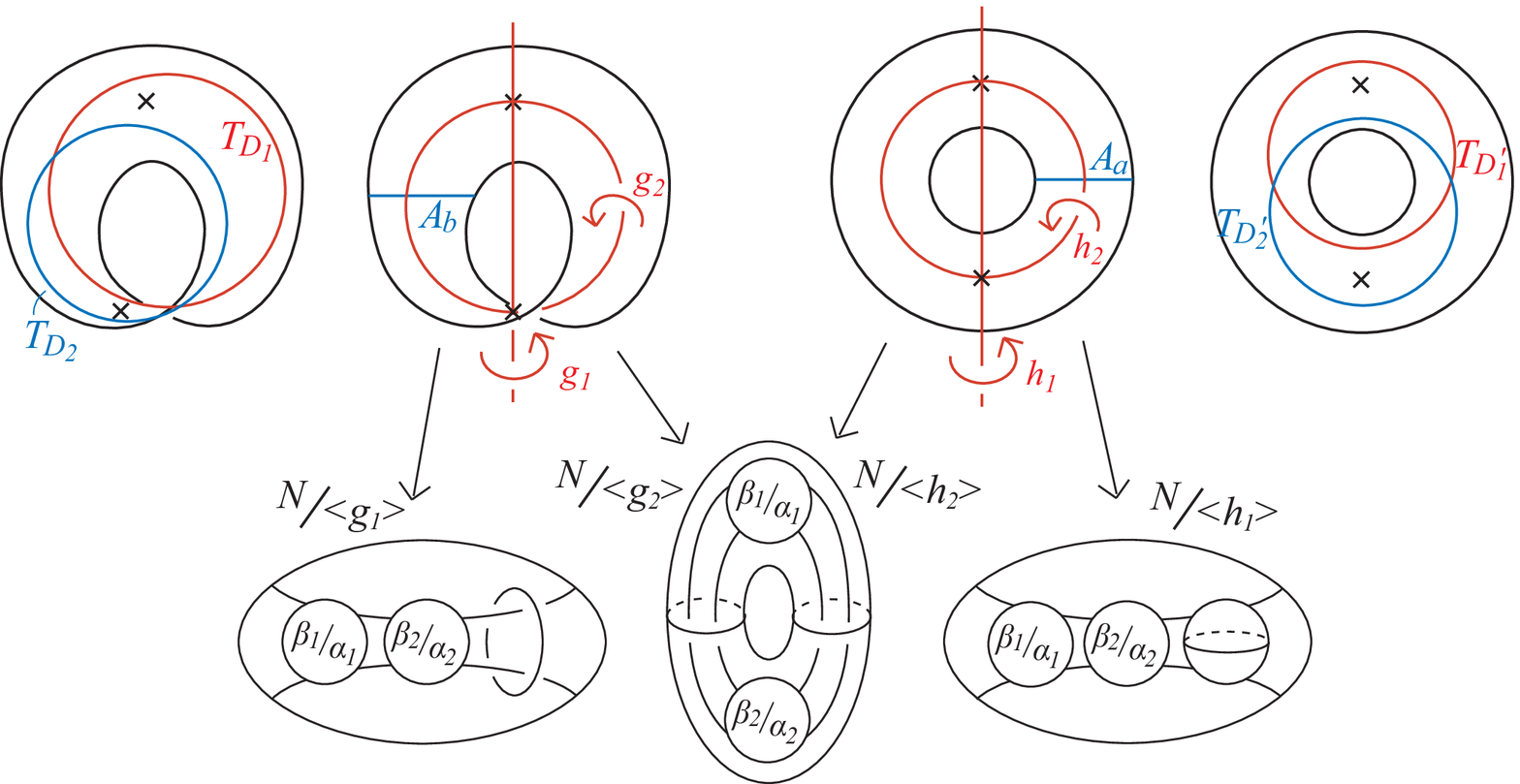}
\end{center}
\caption{}
\label{fig-mcg}
\end{figure}
%
For more precise description of the above elements, see \cite[Section 5 and Remark 2]{Jan2}.

\begin{lemma}\label{lem-mcg-sfs2}
{\rm (1)} If $N$ is a Seifert fibered space $M\ddot{o}(\beta_1/\alpha_1,\beta_2/\alpha_2)\in M\ddot{o}[2]$,
then $\mcg(N)=\langle b\rangle$\\$\times(\langle D_1,D_2\rangle\rtimes\langle g_1,g_2\rangle)$ 
and has a group presentation
\begin{align*}
\begin{array}{rll}
\mcg(N)=\langle D_1,D_2,g_1,g_2,b&|&g_i^2, [g_1,g_2], g_1D_jg_1=D_j^{-1}, g_2D_1g_2=D_2^{-1},\\
&&b^2, [g_i,b], [D_j,b]\ (i,j\in\{1,2\})\rangle.
\end{array}
\end{align*}
In particular, the subgroup $\langle D_1,D_2\rangle$ of $\mcg(N)$
is a free group of rank 2.

{\rm (2)} If $N$ is a Seifert fibered space $A(\beta_1/\alpha_1,\beta_2/\alpha_2)\in A[2]$,
then $\mcg(N)=\langle a\rangle\rtimes(\langle D_1',D_2'\rangle\rtimes$\\$\langle h_1,h_2\rangle)$ 
and has a group presentation
\begin{align*}
\begin{array}{rll}
\mcg(N)=\langle D_1',D_2',h_1,h_2,a&|&h_i^2, [h_1,h_2], h_1D_j'h_1=D_j'^{-1}, h_2D_1'h_2=D_2'^{-1},\\
&&h_1ah_1=a, h_2ah_2=a^{-1}, D_j'aD_j'^{-1}=a\ (i,j\in\{1,2\})\rangle.
\end{array}
\end{align*}
In particular, the subgroup $\langle D_1',D_2'\rangle$ of $\mcg(N)$
is a free group of rank 2.
\end{lemma}

\begin{proof}
By \cite[Proposition 25.3]{Joh}, we have a split exact sequence
$$
1\rightarrow \mcg^0(N)\rightarrow \mcg(N)\rightarrow \mcg^{\ast}(F)\rightarrow 1.
$$

(1) By \cite[Lemma 25.2]{Joh}, $\mcg^0(N)$ is an order-2 group generated by $b$.
On the other hand, by \cite[Section 4.1]{Bir3},
we have the following exact sequence, called the \lq\lq Birman exact sequence\rq\rq.
$$
1\rightarrow \pi_1(F',x_0)\rightarrow \mcg^{\ast}(F)(\cong \mcg_2)\rightarrow \mcg_1\rightarrow 1,
$$
where 
$\pi_1(F',x_0)$ denotes the fundamental group of a once-punctured M\"{o}bius band and
$\mcg_{n}$ denotes the mapping class group of a M\"{o}bius band fixing $n$ specified points.
Recall that $\pi_1(F',x_0)$ is a free group of rank 2
and that $\mcg_1=\langle g_1,g_2 \rangle\cong\Z_2\oplus\Z_2$ (cf. \cite[Lemma 4]{Jan2}).
Moreover, we may take the images of $T_{D_1}$ and $T_{D_2}$ by the projection map as the generators of $\pi_1(F',x_0)$.
Then their images in $\mcg^{\ast}(F)$, by the second map in the above exact sequence, are $D_1$ and $D_2$.
Moreover, the conjugation of $D_j$ by $g_i$ $(i,j\in\{1,2\})$ is as follows:
$$
g_1D_jg_1=D_j^{-1},\ \ g_2D_1g_2=D_2.
$$
Hence, by using an argument in \cite[p.136--139]{Joh3},
we obtain the following group presentation of $\mcg^{\ast}(F)$.
$$
\mcg^{\ast}(F)=\langle D_1,D_2,g_1,g_2\mid g_i^2, [g_1,g_2], 
g_1D_jg_1=D_j^{-1},g_2D_1g_2=D_2\ (i,j\in\{1,2\})\rangle.
$$
Since the conjugation of $b$ by $D_j$ or $g_1$ is $b$,
we obtain the desired result by using an argument in \cite[p.136--139]{Joh3} again.

(2) can be proved similarly.
%
\end{proof}
\vspace{2mm}

Let $M$ be a manifold in Theorem \ref{thm-kob},
and $T$ the union of tori 
in the theorem.
Let $\D$ be the subgroup of $\mcg(M)$ generated by the all possible Dehn twists along $T$.
Then we obtain the following,
which can be proved by an argument similar to that for \cite[Proposition 15.2]{Bon} or \cite[Lemma 3]{Jan2}.

\begin{lemma}\label{lem-d}
Let $M$ be a manifold in Theorem \ref{thm-kob}.

{\rm (1)} If $M$ satisfies the condition (M1) or (M2), 
then $\D$ is an infinite cyclic group generated by $D_l$,
where $D_l$ is the Dehn twist along (a component of) $T$ in the direction of a longitude of $K$.

{\rm (2)} If $M$ satisfies the condition (M3-1-1),
then $\D\cong \langle D_l\rangle\cong\Z$, 
where $D_l$ is the Dehn twist along (a component of) $T$ in the direction of a longitude of $K$.

{\rm (3)} If $M$ satisfies the condition (M3-1-2), 
then $\D$ is generated by $D_m$ and $D_l$, 
where $D_m$ and $D_l$ are the Dehn twists along $T$ in the direction of a meridian and a longitude of $K$, respectively.
Moreover, $\D\cong\langle D_m\rangle\cong\Z$ if $M_1\in M\ddot{o}[0]$, and 
$\D\cong \langle D_m, D_l\rangle\cong\Z\oplus\Z$ otherwise.

{\rm (4)} If $M$ satisfies the condition (M3-2-1), namely, $M$ is obtained by gluing $M_1\in A[r]$ ($r=0,1,2$) and $M_2=E(K)=A(1/n)$ so that the regular fibers of $M_1$ are identified with the meridians of $K$,
then $\D$ is an abelian group generated by $D_m$ and $D_l$, 
where $D_m$ and $D_l$ are the Dehn twists along (a component of) $T$ in the direction of a meridian and a longitude of $K$, respectively.

{\rm (5)} If $M$ satisfies the condition (M3-2-2), 
then $\D$ is generated by $D_{m_1}, D_{l_1}$ and $D_{l_2}$,
where $D_{m_i}$ and $D_{l_i}$ are the Dehn twists along a component $T_i$ ($i=1,2$) of $T$ in the direction of a meridian and a longitude of $K$, respectively.
Moreover, $\D\cong \langle D_{m_1}, D_{l_1}\rangle\cong\Z\oplus\Z$ if $M_1\in A[0]$, and
$\D\cong \langle D_{m_1}, D_{l_1}, D_{l_2}\rangle\cong\Z\oplus\Z\oplus\Z$ otherwise.

{\rm (6)} If $M$ satisfies the condition (M4), 
then $\D\cong \langle D_{l_1}, D_{l_2}\rangle\cong\Z\oplus\Z$, 
where $D_{l_i}$ is the Dehn twist along a component $T_i$ ($i=1,2$) of $T$ in the direction of a longitude of $K$.
\end{lemma}


We define some self-homeomorphisms of $M$ when $M$ satisfies the condition (M3-1-1) or (M3-2-1) as follows.

\begin{definition}\label{def-homeo}
{\rm
Let $M$ be a manifold which satisfies the condition (M3-1-1) or (M3-2-1) with $r=2$. 

(1) When $M$ satisfies the condition (M3-1-1),
we define self-homeomorphisms $G_1$, $G_2$ and $B$ of $M$ as follows.
\begin{eqnarray*}
\begin{array}{lll}
G_1|_{M_1}=g_1, & G_1|_{M_2}=f, & G_1|_{T\times[1,2]}=R,\\
G_2|_{M_1}=g_2, & G_2|_{M_2}=id, & G_2|_{T\times[1,2]}=R_lD^{1/2}_{l},\\
B|_{M_1}=b, & B|_{M_2}=id, & B|_{T\times[1,2]}=R_mD^{1/2}_{m},\\
D_j|_{M_1}=D_j, & D_j|_{M_2}=id, & D_j|_{T\times[1,2]}=id\hspace{5mm}(j=1,2).
\end{array}
\end{eqnarray*}
Here, $g_1$, $g_2$ and $b$ are involutions of $M_1$ 
as described in Lemma \ref{lem-mcg-sfs2}, 
$f$ is an involution of $M_2=E(K)$ which gives a strong inversion of the torus knot $K$ (see \cite[Remark 7]{Jan2}),
and $R$ and $R_{\alpha}$ ($\alpha=m$ or $l$) are the self-homeomorphisms of $T\times [1,2]$ defined by
$R([\vec{x}],t)=([-\vec{x}],t)$ and $R_{\alpha}([\vec{x}],t)=([\vec{x}+\frac{1}{2}\vec{\alpha}],t)$, respectively.
Here, we identify $T$ with $\R^2/\Z^2$ and
$[\vec{x}]$ denotes the point of $\R^2/\Z^2$ determined by $\vec{x}\in \R^2$.
In the identity $D_j|_{M_1}=D_j$,
the right-hand side represents the homeomorphisms in Lemma \ref{lem-mcg-sfs2} (1).
The symbols $D_l$ and $D_m$ denote the Dehn twists given in Lemma \ref{lem-d} (2) and (4),
and $D^{1/2}_l$ and $D^{1/2}_m$ denote the half Dehn twists in the direction of $l$ and $m$, respectively
(see \cite[Section 5]{Jan2} for definition of half Dehn twists).
\vspace{2mm}

(2) When $M$ satisfies the condition (M3-2-1),
we define self-homeomorphisms $H_1$ and $H_2$ of $M$ as follows.
\begin{eqnarray*}
\begin{array}{lll}
H_1|_{M_1}=h_1, & H_1|_{M_2}=h_1, & H_1|_{T\times[1,2]}=R,\\
H_2|_{M_1}=h_2, & H_2|_{M_2}=h_2, & H_2|_{T\times[1,2]}=h_2|_T\times[1,2],\\
D_j'|_{M_1}=D_j', & D_j'|_{M_2}=id, & D_j'|_{T\times[1,2]}=id.
\end{array}
\end{eqnarray*}
Here, $h_1$ is an involution of $M_1$ as described in Lemma \ref{lem-mcg-sfs2},
$h_2$ is an involution of $M_2=E(K)$ which gives a strong inversion of the torus link $K$ (see \cite[Lemma 4 (3)]{Jan2}),
and $R$ is the self-homeomorphism of $T\times [1,2]$ defined in (1).
In the identity $D_j'|_{M_1}=D_j'$,
the right-hand side represents the homeomorphisms in Lemma \ref{lem-mcg-sfs2} (2).
}
\end{definition}

In \cite[Proposition 6]{Jan2},
the author calculated $\mcg(M)$ for certain manifolds in Theorem \ref{thm-kob}
by using \cite[Theorem 15.1]{Bon} and \cite{Sak3}.
The following theorem can be obtained by a similar argument
together with Lemmas \ref{lem-mcg-sfs2} and \ref{lem-d}.

\begin{theorem}\label{thm-mcg}
Let $M$ be a manifold which satisfies the condition (M3-1-1) or (M3-2-1) with $r=2$. 

(1) If $M$ satisfies the condition (M3-1-1),
then the subgroup $\langle D_1,D_2\rangle$ of $\mcg(M)$ generated by $D_1$ and $D_2$ 
is a free group of rank 2.

(2) If $M$ satisfies the condition (M3-2-1),
then the subgroup $\langle D_1',D_2',D_l\rangle$ of $\mcg(M)$ 
is the direct product of the infinite cyclic group generated by $D_l$ and the free group of rank 2 generated by $D_1'$ and $D_2'$.
\end{theorem}

\begin{proof}
Recall from \cite[Theorem 15.1]{Bon} and \cite{Sak3} (cf. \cite[Section 5]{Jan2})
that there is an exact sequence
\begin{equation}\label{exact}
1\rightarrow \D\rightarrow \mcg(M)\rightarrow \Delta\rightarrow 1,
\end{equation}
where $\Delta$ is the subgroup of $\mcg(M_1)\times \mcg(M_2)$ consisting of all elements $(f_1,f_2)$
such that $f_1|_T$ is isotopic to $f_2|_T$.

(1) Let $M$ be a manifold which satisfies the condition (M3-1-1).
Then we have $\mcg(M_1)=$\\$\langle b\rangle\times(\langle D_1,D_2\rangle\rtimes\langle g_1,g_2\rangle)$ by Lemma \ref{lem-mcg-sfs2} (1), and $\mcg(M_2)=\langle f\rangle\cong \Z_2$ by \cite[Lemma 4 (1)]{Jan2}.
Note that the subgroup of $\mcg(M_1)$ generated by $D_1$ and $D_2$ is a free group of rank 2.
Hence, we see that 
$$\Delta=\langle (b,id)\rangle\times(\langle (D_1,id),(D_2,id)\rangle\rtimes\langle (g_1,f),(g_2,id)\rangle)$$ 
and the subgroup of $\Delta$ generated by $(D_1,id)$ and $(D_2,id)$ is also a free group of rank 2.
Hence, we see from the exact sequence (\ref{exact}) that the subgroup of $\mcg(M)$ generated by $D_1$ and $D_2$ is a free group of rank 2.

(2) Let $M$ be a manifold which satisfies the condition (M3-2-1).
Then we have $\mcg(M_1)=$\\$\langle a\rangle\rtimes(\langle D_1',D_2'\rangle\rtimes\langle h_1,h_2\rangle)$ by Lemma \ref{lem-mcg-sfs2} (2).
On the other hand, we have $\mcg(M_2)=\langle a\rangle\rtimes\langle h_1,$\\$h_2\rangle$ (cf. \cite[Lemma 4 (2)]{Jan2}).
Note that the subgroup of $\mcg(M_1)$ generated by $D_1'$ and $D_2'$ is a free group of rank 2.
We see that 
$$\Delta=\langle (D_1',id),(D_2',id)\rangle\rtimes\langle (h_1,h_1),(h_2,h_2)\rangle$$ 
and the subgroup of $\Delta$ generated by $(D_1',id)$ and $(D_2',id)$ is also a free group of rank 2.
Hence, we see from the exact sequence (\ref{exact}) that the subgroup of $\mcg(M)$ generated by $D_1'$ and $D_2'$ is a free group of rank 2.
On the other hand, the subgroup of $\mcg(M)$ generated by $D_l$ is an infinite cyclic group by Lemma \ref{lem-d} (3).
Since we can easily see that $D_l$ commutes with $D_1'$ and $D_2'$,
we obtain the desired result.
\end{proof}

\begin{remark}\label{rmk-mcg}
{\rm
Let $M$ be a manifold in Theorem \ref{thm-mcg}. 

(1) If $M$ satisfies the condition (M3-1-1), 
then $\mcg(M)=\langle B\rangle\rtimes(\langle D_1,D_2\rangle\rtimes\langle G_1,G_2\rangle)$, and has a group presentation
\begin{align*}
\begin{array}{rll}
\mcg(M)=\langle D_1,D_2,G_1,G_2,B
&|&
G_i^2, [G_1,G_2], G_1D_jG_1=D_j^{-1}, G_2D_1G_2=D_2^{-1},\\
&&B^2=D_m, [G_1,B]=D_m^{-1}, [G_2,B], [D_j,B]
\ (i,j\in\{1,2\})\rangle.
\end{array}
\end{align*}

(2) If $M$ satisfies the condition (M3-2-1), 
then $\mcg(M)=\langle D_m,D_l\rangle\rtimes(\langle D_1',D_2'\rangle\rtimes\langle H_1,H_2\rangle)$, and has a group presentation
\begin{align*}
\begin{array}{rll}
\mcg(M)=\langle D_1',D_2',H_1,H_2,D_m, D_l&|&
H_i^2, [H_1,H_2], H_1D_j'H_1=D_j'^{-1}, H_2D_1'H_2=D_2'^{-1},\\
&&D_j'D_mD_j'^{-1}=D_m, D_j'D_lD_j'^{-1}=D_l,\\
&&H_1D_mH_1=D_m^{-1}, H_2D_mH_2=D_m, H_iD_lH_i=D_l^{-1}\\
&&(i,j\in\{1,2\})\rangle.
\end{array}
\end{align*}
}
\end{remark}


\section{Proof of Theorem \ref{thm-main}}\label{sec-proof}

Since the if part is already proved in \cite{Jan},
we prove the only if part.
Namely, we show that any prime, unsplittable 3-bridge link 
which admits infinitely many 3-bridge spheres up to isotopy
is equivalent to a link $L(q/2p;\beta_1/\alpha_1,\beta_2/\alpha_2)$ in Figure \ref{fig-3b-inf} with $q\not\equiv 1\pmod{p}$ and $|\alpha_1|>1$ (or $|\alpha_2|>1$).

Let $L$ be a prime, unsplittable 3-bridge link in $S^3$, and
assume that $L$ admits infinitely many 3-bridge spheres up to isotopy.
Let $M=M_2(L)$ be the double branched cover of $S^3$ branched along $L$
and $\tau_L$ the covering transformation.
By Proposition \ref{prop-hs-3b}, 
$M$ admits infinitely many genus-2 Heegaard surfaces, up to isotopy,
whose hyper-elliptic involutions are $\tau_L$.
By \cite[Theorem 1.1]{Li}, $M$ is toroidal, and hence,
either $M$ is a Seifert fibered space or $M$ admits a nontrivial torus decomposition.

\begin{case}\label{case-1}
$M$ is a Seifert fibered space.
\end{case}

By the orbifold theorem \cite{Boi7, Coo} together with \cite[Section 5]{Dun}, 
$L$ is a {\it generalized Montesinos link} or a {\it Seifert link}, 
that is, either $L$ is equivalent to a link in Figure \ref{fig-gen-mont} or 
$S^3\setminus L$ admits a Seifert fibration.
%
\begin{figure}
\begin{center}
\includegraphics*[width=6.2cm]{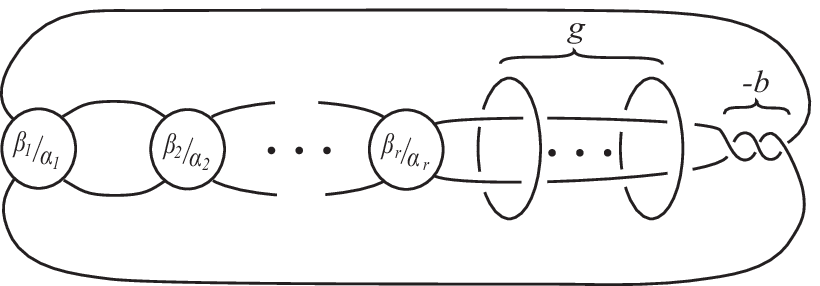}
\end{center}
\caption{}
\label{fig-gen-mont}
\end{figure}
%

Assume first that $L$ is a generalized Montesinos link.
By \cite[Theorem 2.1]{Boi3},
$L$ is equivalent to one of the links in Figure \ref{fig-case1-1} since $L$ is a 3-bridge link.
%
\begin{figure}
\begin{center}
\includegraphics*[width=10.3cm]{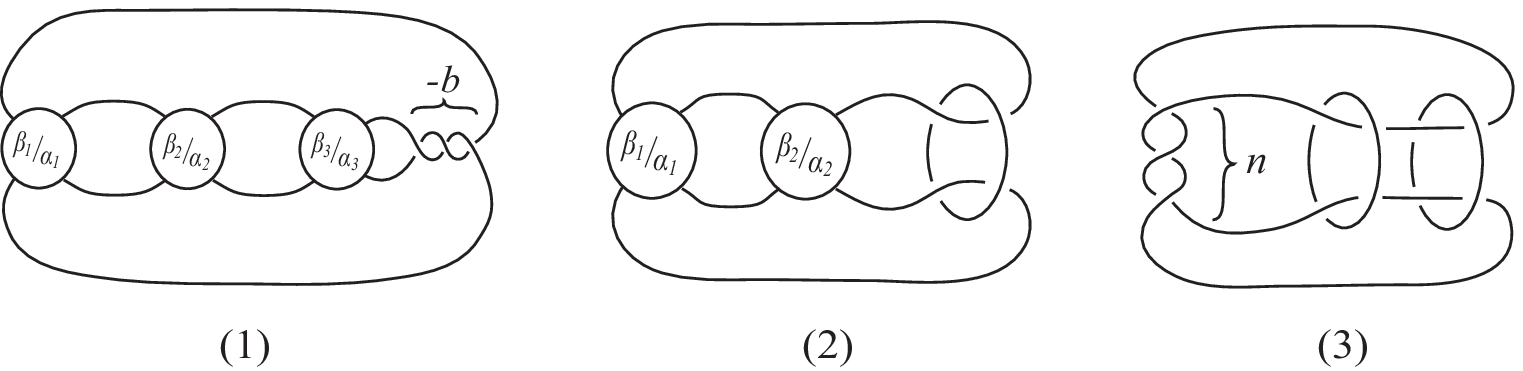}
\end{center}
\caption{}
\label{fig-case1-1}
\end{figure}
%
By \cite{Jan2,Jan3}, 
$L$ admits at most six 3-bridge spheres, at most two 3-bridge spheres or a unique 3-bridge sphere up to isotopy
according as $L$ is equivalent to the link in Figure \ref{fig-case1-1} (1), (2) or (3).
This contradicts the assumption that $L$ admits infinitely many 3-bridge spheres up to isotopy.

Next, assume that $L$ is a Seifert link.
By \cite{Bur2} and by the assumption that $L$ is a 3-bridge link, 
we see that $L$ is equivalent to a (nontrivial) $(3,n)$-torus link
or the union of a $(2,n)$-torus knot and its core of index 2.

If $L$ is equivalent to a $(3,n)$-torus knot 
or the union of a $(2,n)$-torus knot and its core of index 2,
then $M$ is a small Seifert fibered space
and admits at most four genus-2 Heegaard surfaces up to isotopy by \cite{Boi},
which is a contradiction.
We also prove the following proposition in Section \ref{sec-toruslink}.

\begin{proposition}\label{prop-toruslink}
If $L$ is a $(3,3n')$-torus link for some nonzero integer $n'$, 
then $L$ admits a unique 3-bridge sphere up to isotopy.
\end{proposition}

Hence, $M$ cannot be a Seifert fibered space.

\begin{case}\label{case-2}
$M$ admits a nontrivial torus decomposition.
\end{case}

If $L$ is an arborescent link, 
then $L$ admits at most four 3-bridge spheres by \cite{Jan3}.
Hence, we assume that $L$ is not an arborescent link.
Then, by Theorem \ref{thm-kob} and \cite[Proof of Theorem 1]{Jan2}, 
$M$ satisfies one of the conditions (M1), (M2), (M3-1-1), (M3-1-2), (M3-2-1), (M3-2-2) and (M4)
introduced at the end of Section \ref{sec-hs}.
Let $T$ be the union of tori as in Theorem \ref{thm-kob}.

\begin{subcase2}\label{case2-1}
$M$ satisfies the condition (M1).
\end{subcase2}

Note that $M$ is obtained by gluing $M_1\in D[2]$ and $M_2=L(p,q)\setminus N(K)$, 
where $K$ is a 1-bridge knot in a lens space $L(p,q)$, 
and that $M_2$ is hyperbolic.
Since $M_1$ is also simple, 
we can see that $T=\partial M_1=\partial M_2$ is the only essential torus in $M$ up to isotopy.
By \cite[Theorem 4]{Joh2}, 
there exist genus-2 Heegaard surfaces $F_1,F_2,\dots,F_n$ of $M$ such that
any genus-2 Heegaard surface $F$ can be obtained from some $F_i$
by applying Dehn twists along $T$.
Recall from Theorem \ref{thm-kob} that
$F\cap M_1$ is an essential saturated annulus of $M_1$ and 
$F\cap M_2$ is a 2-hold torus which gives a 1-bridge presentation of $K$.
Let $\mu$ and $\lambda$ be the meridian and a longitude of $K$,
and denote the Dehn twist along $T$ in the direction of $\mu$ and $\lambda$ 
by $D_{\mu}$ and $D_{\lambda}$, respectively.
Then $F$ is isotopic to $D_{\mu}^{n_1}D_{\lambda}^{n_2}(F_i)$ for some integers $n_1$ and $n_2$.
Note that any genus-2 Heegaard surface meets $T$ in the union of two meridians of $K$.
Hence, $D_{\mu}^{n_1}D_{\lambda}^{n_2}(F_i)=D_{\lambda}^{n_2}(F_i)$.
Note also that $\tau_{F_j}D_{\lambda}\tau_{F_j}=D_{\lambda}^{-1}$ 
because $\tau_{F_j}$ reverses the orientation of $\lambda$ (see Figure \ref{fig-case2-1}).
%
\begin{figure}
\begin{center}
\includegraphics*[width=8.3cm]{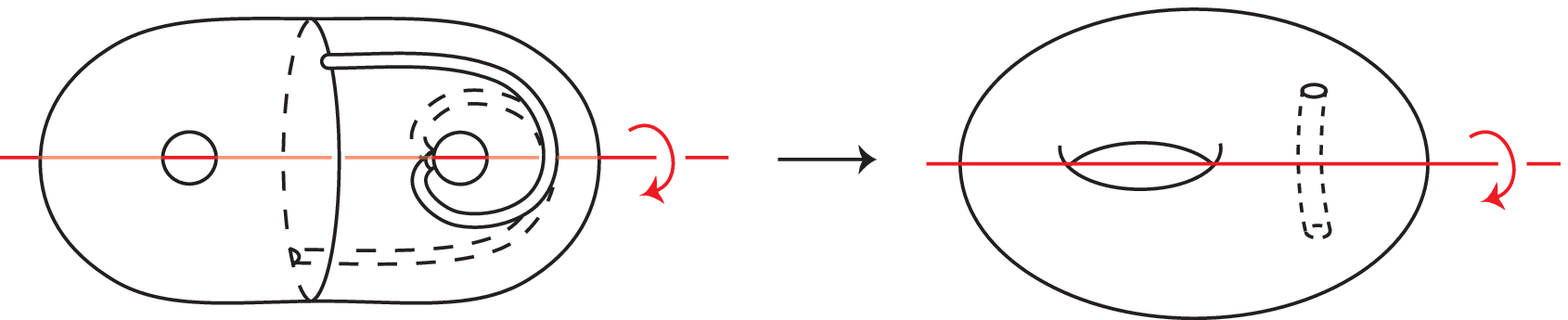}
\end{center}
\caption{}
\label{fig-case2-1}
\end{figure}
%
Thus,
$$
\tau_F=\tau_{D_{\lambda}^{n_2}(F_j)}=D_{\lambda}^{n_2}\tau_{F_j}D_{\lambda}^{-n_2}=D_{\lambda}^{2n_2}\tau_{F_j}.
$$
Since $\D$ is an infinite cyclic group generated by $D_{\lambda}$ by Lemma \ref{lem-d} (1),
$\{D_{\lambda}^{2n_2}\tau_{F_j}\}_{n_2\in\Z}$ are mutually distinct,
and hence, there is at most one $n_2\in\Z$ such that 
$D_{\lambda}^{2n_2}\tau_{F_j}=\tau_{L}$.
So, for each Heegaard surface $F_j$,
the hyper-elliptic involution associated with $D_{\lambda}^{n_2}(F_j)$ 
is strongly equivalent to $\tau_L$ 
for at most one $n_2\in\Z$.
Hence, the number of genus-2 Heegaard surfaces whose hyper-elliptic involutions are $\tau_L$
is finite.
This contradicts the assumption.

\begin{subcase2}\label{case2-2}
$M$ satisfies the condition (M2).
\end{subcase2}

Note that $M$ is obtained by gluing $M_1\in D[2]\cup D[3]$ and $M_2=E(K)$, where $K$ is a hyperbolic 2-bridge knot,
so that the regular fiber of $M_1$ is identified with the meridian loop of $K$.
Recall from Theorem \ref{thm-kob} that,
for a genus-2 Heegaard surface $F$,
$F\cap M_1$ is the union of two essential annuli which cuts $M_1$ into three solid tori and
$F\cap M_2$ is the 2-bridge sphere of $K$.
Let $D_{\lambda}$ denote the Dehn twist along $T=\partial M_1=\partial M_2$ in the direction of a longitude of $K$, and
note that $\D$ is an infinite cyclic group generated by $D_{\lambda}$ (see Lemma \ref{lem-d} (1)).

First assume that $M_1\in D[2]$.
Note that $K$ admits a unique 2-bridge sphere up to isotopy by \cite{Sch},
and that $M_1$ contains a unique essential annulus up to isotopy. 
By \cite[Lemma 6]{Jan2}, 
there exist genus-2 Heegaard surfaces $F_0$, $F_1$, $F_2$ and $F_3$ of $M$ such that
any genus-2 Heegaard surface of $M$ is isotopic to $D_{\lambda}^n(F_i)$ for some integer $n$ and for some $i=0,1,2,3$.
(We remark that $F_i=D_{\lambda}^{i/4}(F_0)$.)
Recall that $\tau_{F_i}D_{\lambda}\tau_{F_i}=D_{\lambda}^{-1}$.
By an argument similar to that in Case 2.\ref{case2-1}, 
we can see that the number of genus-2 Heegaard surfaces whose hyper-elliptic involutions are $\tau_L$ is finite, a contradiction.

Next, assume that $M_1\in D[3]$.
Note that $F\cap M_1$ is homeomorphic to one of $G_1,G_2$ and $G_3$ in Figure \ref{fig-annuli}.%
\begin{figure}
\begin{center}
\includegraphics*[width=9cm]{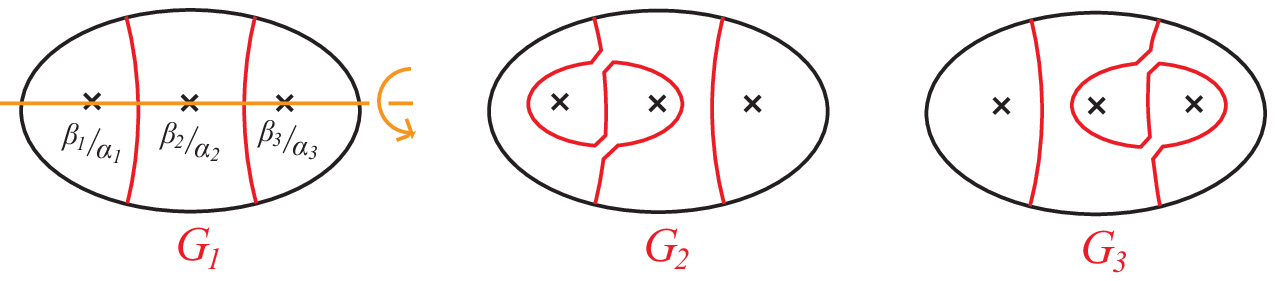}
\end{center}
\caption{}
\label{fig-annuli}
\end{figure}
%
To be precise, $F\cap M_1$ is isotopic to $f_1(G_i)$ for some $f_1\in\mcg(M_1)$ and for some $i=1,2,3$.
(We may assume that $f_1|_{\partial M_1}=id$.)
For each $i=1,2,3$,
let $F_i$ be a genus-2 Heegaard surface such that $F_i\cap M_1=G_i$ and $F_i\cap M_2$ is the 2-bridge sphere of $K$.
By \cite[Lemma 6 (1)]{Jan2}, 
any genus-2 Heegaard surface $F$ is isotopic to $D_{\lambda}^{n/4}f(F_i)$ 
for some integer $n$ and for some $i=1,2,3$ and for some homeomorphism $f$ of $M$
which is obtained from some $f_1\in\mcg(M_1)$ by the rule $f|_{M_1}=f_1\in\mcg(M_1)$ and $f|_{M_2}=id$.
Let $F_i^j$ ($j=0,1,2,3$) be the Heegaard surface $D_{\lambda}^{j/4}(F_i)$.
Let $\mcg_0(M)$ be the subgroup of $\mcg(M)$ consisting of all elements whose restrictions to $M_2$ are the identity.
Then the above argument implies that
$F$ is isotopic to $g(F_i^j)$ for some $g\in\mcg_0(M)$ and for some $F_i^j$.

\begin{claim}
For each $F_i^j$, at most one of $\{g(F_i^j)\}_{g\in\mcg_0(M)}$ 
can have $\tau_L$ as hyper-elliptic involution.
\end{claim}

\begin{proof}
We show this only for $F_1^0$. (The other cases can be treated similarly.)
Put $\tau:=\tau_{F_1^0}$.
Then $\tau_{g(F_1^0)}=g\tau g^{-1}$.
Recall by \cite[Proof of Theorem 2 (3)]{Jan2} that
$$
\mcg(M_1)\cong (P_3/\langle (xy)^3\rangle)\rtimes \langle \tau\rangle
< (B_3/\langle (xy)^3\rangle)\rtimes \langle \tau\rangle,
$$
where $P_3$ and $B_3$ are the pure 3-braid group and the 3-braid group, respectively.
Let $\mcg_0(M_1)$ be the subgroup of $\mcg(M_1)$ consisting of all elements whose restrictions to $T$ are the identity.
Then we have $\mcg_0(M_1)\cong P_3/\langle (xy)^3\rangle$.
Recall from \cite[Theorem 15.1]{Bon} and \cite{Sak3} (cf. \cite[Section 5]{Jan2})
that there is an exact sequence
$$
1\rightarrow \D\rightarrow \mcg(M)\rightarrow \Delta\rightarrow 1,
$$
where $\Delta$ is the subgroup of $\mcg(M_1)\times \mcg(M_2)$ consisting of all elements $(f_1,f_2)$
such that $f_1|_T$ is isotopic to $f_2|_T$.
Since $\mcg_0(M)$ is the subgroup of $\mcg(M)$ consisting of all elements whose restrictions to $M_2$ are the identity,
we obtain an exact sequence
$$
1\rightarrow \D\rightarrow \mcg_0(M)\rightarrow \mcg_0(M_1)\rightarrow 1.
$$
Recall from \cite[Claim 1 (2)]{Jan2} that 
the \lq\lq centralizer\rq\rq\  $$Z(\tau,\mcg_0(M_1))=\{f\in\mcg_0(M_1)\mid f\tau=\tau f\}$$ of $\tau$ in $\mcg_0(M_1)$ is $\{1\}$.
By using this fact, the identity $D_{\lambda}\tau D_{\lambda}^{-1}=D_{\lambda}^{2}\tau$ and Lemma \ref{lem-d} (1),
we can see that the \lq\lq centralizer\rq\rq\  $Z(\tau,\mcg_0(M))=\{f\in\mcg_0(M)\mid f\tau=\tau f\}$ of $\tau$ in $\mcg_0(M)$ is $\{1\}$.
This implies that 
the hyper-elliptic involution associated with $g(F_1^0)$ 
is strongly equivalent to $\tau_L$ for at most one $g\in\mcg_0(M)$.
\end{proof}
\vspace{2mm}

Hence, the number of genus-2 Heegaard surfaces of $M$ whose hyper-elliptic involutions are $\tau_L$ 
is at most twelve,
a contradiction.

\begin{subcase2}\label{case2-3-1-1}
$M$ satisfies the condition (M3-1-1).
\end{subcase2}

Recall that $M$ is obtained by gluing $M_1\in M\ddot{o}[r]$ $(r=1,2)$ and $M_2=E(K)$, where $K$ is a $(2,n)$-torus knot,
so that the regular fiber of $M_1$ is identified with the meridian loop of $K$.
By Theorem \ref{thm-kob}, for any genus-2 Heegaard surface $F$, 
$F\cap M_1$ is the union of two essential saturated annuli which cuts $M_1$ into two solid tori
and $F\cap M_2$ is a 2-bridge sphere of $K$.

Assume first that $r=1$.
Note that $M_1$ contains a unique essential saturated annulus up to isotopy
and that $K$ admits a unique 2-bridge sphere up to isotopy preserving $K$ (see \cite[Theorem 4]{Mor}).
Let $\mu$ and $\lambda$ be the meridian and a longitude of $K$, respectively.
Let $F_0$ be a genus-2 Heegaard surface of $M$.
Then, by \cite[Lemma 6 (2)]{Jan2},
any genus-2 Heegaard surface $F$ of $M$ is isotopic to 
$D_{\lambda}^{n/4}(F_0)$ for some integer $n$.
Note that $\D$ is the infinite cyclic group generated by $D_{\lambda}$ (see \cite[Lemma 3]{Jan2}).
Hence, by an argument similar to that in the previous cases, 
we see that the number of genus-2 Heegaard surfaces whose hyper-elliptic involutions are $\tau_L$ is finite, a contradiction.

Assume first that $r=2$.
Pick a \lq\lq standard\rq\rq\  genus-2 Heegaard surface $F_0$ of $M$,
such that $F_0\cap M_1$ is preserved by the homeomorphisms $g_1$, $g_2$ and $b$ in Lemma \ref{lem-mcg-sfs2} (1).
Then we may assume that $\tau_{F_0}(=\tau_L)=G_2$.
By Theorem \ref{thm-mcg} (1) and \cite[Lemma 6 (2)]{Jan2},
any genus-2 Heegaard surface $F$ of $M$ is isotopic to $D_1^{n_1}D_2^{n_2}\cdots D_1^{n_{2m-1}}D_2^{n_{2m}}(F_0)$
for some integers $n_1, n_2,\dots,n_{2m-1}$ and $n_{2m}$,
where $n_2,\dots,n_{2m-1}$ are nonzero.
Then $$\tau_F=D_1^{n_1}D_2^{n_2}\cdots D_1^{n_{2m-1}}D_2^{n_{2m}}G_2D_2^{-n_{2m}}D_1^{n_{2m-1}}\cdots D_2^{-n_2}D_1^{-n_1}.$$
Since $G_2D_1G_2=D_2^{-1}$ by Remark \ref{rmk-mcg},
we have $$\tau_F=D_1^{n_1}D_2^{n_2}\cdots D_1^{n_{2m-1}}D_2^{n_{2m}}D_1^{n_{2m}}D_2^{n_{2m-1}}\cdots D_1^{n_2}D_2^{n_1}G_2.$$
By Theorem \ref{thm-mcg} (1), we can see that $\tau_F=\tau_{F_0}$ implies $n_1=n_2=0$,
which means $F$ is isotopic to $F_0$.
This contradicts the assumption.

\begin{subcase2}\label{case2-3-1-2}
$M$ satisfies the condition (M3-1-2).
\end{subcase2}

Note that $M$ is obtained by gluing $M_1\in M\ddot{o}[r]$ $(r=0,1,2)$ and $M_2=E(K)$, 
where $K=S(p,q)$ is a hyperbolic 2-bridge knot,
so that the regular fiber of $M_1$ is identified with the meridian loop of $K$.
We may assume that $q$ is odd.
Let $F$ be a genus-2 Heegaard surface of $M$.
By Theorem \ref{thm-kob},
$F\cap M_1$ is the union of two essential saturated annuli which cuts $M_1$ into two solid tori
and $F\cap M_2$ is a 2-bridge sphere of $K$. 
Note that $M_1/\langle \tau_F\rangle$ is a solid torus and
that the image of $\fix\tau_F\cap M_1$ forms a link in it
as illustrated in Figure \ref{fig-quotient1}.
%
\begin{figure}
\begin{center}
\includegraphics*[width=8.3cm]{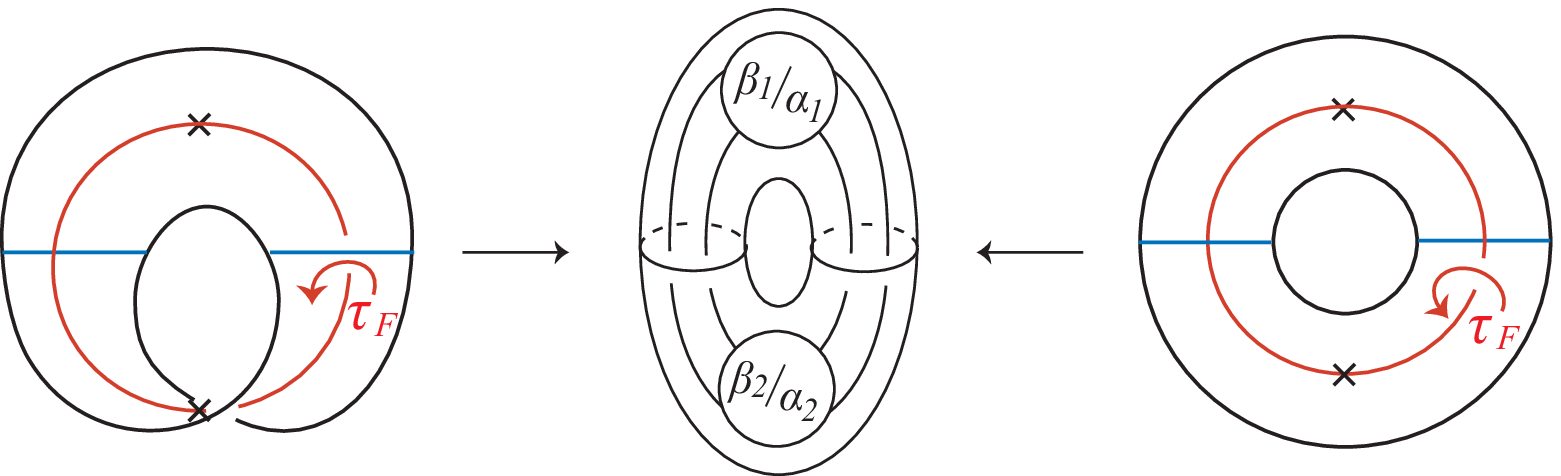}
\end{center}
\caption{}
\label{fig-quotient1}
\end{figure}
%
On the other hand, $M_2/\langle \tau_F\rangle$ is also a solid torus
and the image of $\fix\tau_F\cap M_2$ forms a knot in it
such that its exterior in $M_2$ is the exterior of a 2-bridge link $S(2p,q)$ in $S^3$ (cf. \cite[Lemma 3.2]{Jan}).
Since the meridian and the longitude of the solid torus $M_1/\langle \tau_F\rangle$ are identified with the longitude and the meridian of the solid torus $M_2/\langle \tau_F\rangle$, respectively,
we see that $L$ is equivalent to $L(q/2p;\beta_1/\alpha_1,\beta_2/\alpha_2)$.
Moreover, since $K=S(p,q)$ is a hyperbolic 2-bridge knot,
we have $q\not\equiv \pm 1\pmod{p}$ by \cite{Men}.

First assume that $r=0$.
Note that $M_1$ has a unique essential saturated annulus up to isotopy and
that $K$ admits a unique 2-bridge sphere up to isotopy.
Let $F_0$ be the pre-image of $P^0$ given in \cite{Jan} (cf. Figure \ref{fig-3b-inf})
by the covering map $M\rightarrow S^3$ (branched over $L$).
By an argument similar to that in the previous cases,
$F$ is isotopic to $D_{\lambda}^{n/4}(F_0)$ for some integer $n$,
where $D_{\lambda}$ is the Dehn twist along a component of $T=\partial M_1=\partial M_2$
in the direction of a longitude of $K$.
However, $F=D_{\lambda}^{n/4}(F_0)$ is isotopic to $F_0$ since $D_{\lambda}^{n/4}(F_0)\cap M_1$ can be isotoped to $F_0\cap M_1$ by an isotopy fixing the boundary of $M_1$ as illustrated in Figure \ref{fig-isotopy}. 
Thus $M$ admits a unique genus-2 Heegaard surface up to isotopy, a contradiction.
%
\begin{figure}
\begin{center}
\includegraphics*[width=6cm]{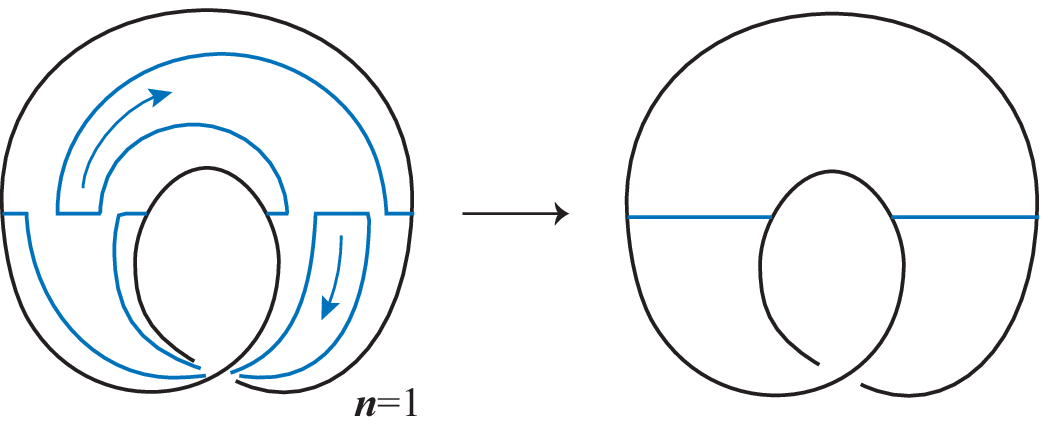}
\end{center}
\caption{}
\label{fig-isotopy}
\end{figure}
%

Assume that $r=1$ or $r=2$.
Then, we see by \cite{Jan}
that 
$L$ is equivalent to a link $L(q/2p;\beta_1/\alpha_1,\beta_2/\alpha_2)$ in Figure \ref{fig-3b-inf} 
and that
$L$ admits infinitely many 3-bridge spheres $\{P^i\}_{i\in\Z}$ up to isotopy.
(Moreover, we can see that any 3-bridge sphere is isotopic to $P^i$ for some $i\in\Z$
by using an argument similar to that in the previous cases and by Lemma \ref{lem-mcg-sfs2}.)

\begin{subcase2}\label{case2-3-2}
$M$ satisfies the condition (M3-2-1) or (M3-2-2).
\end{subcase2}

Note that $M$ is obtained by gluing $M_1\in A[r]$ $(r=0,1,2)$ and $M_2=E(K)$, 
where $K$ is a 2-bridge link,
so that the regular fiber of $M_1$ is identified with the meridian loop of $K$.
Let $F$ be a genus-2 Heegaard surface of $M$.
By Theorem \ref{thm-kob},
$F\cap M_1$ is the union of two essential saturated annuli which cuts $M_1$ into two solid tori
and $F\cap M_2$ is a 2-bridge sphere of $K$. 
(If $M_1$ is homeomorphic to a 2-bridge knot exterior, then $F$ can intersect each $M_i$ so that $F\cap M_1$ is a 2-bridge sphere and $F\cap M_2$ is the union of two essential saturated annuli.)
Pick a \lq\lq standard\rq\rq\  genus-2 Heegaard surface $F_0$ of $M$ and assume that $\tau_{F_0}=\tau_L$.
By using an argument similar to that in the previous cases,
we see the following hold.
\begin{itemize}
\item 
If $M$ satisfies the condition (M3-2-1), where $r=0$ or $1$, 
then $F$ is isotopic to $D_{\lambda}^{n/2}(F_0)$ for some integer $n$,
where $\lambda$ is a longitude or a meridian of $K$ according as $F_0$ meets $T:=\partial M_1=\partial M_2$ in a meridian or a longitude of $K$.
Note that the subgroup of $\D$ generated by $D_{\lambda}$ is finite or an infinite cyclic group.
If the subgroup is finite, then $M$ admits only finitely many genus-2 Heegaard surfaces up to isotopy, a contradiction.
If the subgroup is an infinite cyclic group, then, by an argument similar to that in Cases 2.\,\ref{case2-1} together with Lemma \ref{lem-d} (4), we see that the number of genus-2 Heegaard surfaces of $M$ whose hyper-elliptic involutions are $\tau_L$ is finite, a contradiction.
\item 
If $M$ satisfies the condition (M3-2-1), where $r=2$, 
then we see that $F$ is isotopic to $D_l^{n_0/2}D_1'^{n_1}D_2'^{n_2}\cdots D_1'^{n_{2m-1}}D_2'^{n_{2m}}(F_0)$ for some integers $n_i$ ($i=0,1,\dots, 2m$) by using Theorem \ref{thm-mcg} (2). 
By an argument similar to that in Case 2.\,\ref{case2-3-1-1} together with Theorem \ref{thm-mcg} (2), we see that $\tau_F=\tau_{F_0}(=\tau_L)$ implies $n_i=0$ for all $i=0,1,\dots, 2m$. 
Hence, $F$ is isotopic to $F_0$, a contradiction.
\item 
If $M$ satisfies the condition (M3-2-2), where $r=0$,
then $F$ is isotopic to $D_{\lambda}^{n/2}(F_0)$ for some integer $n$,
where $\lambda$ is a longitude of $K$.
By an an argument similar to that in Cases 2.\,\ref{case2-1} together with Lemma \ref{lem-d} (5),
we see that the number of genus-2 Heegaard surfaces of $M$ whose hyper-elliptic involutions are $\tau_L$ is finite, a contradiction.
\item 
If $M$ satisfies the condition (M3-2-2), where $r=1$ or $2$,
then we see by \cite{Jan} that
the link $L$ is equivalent to a link $L(q/2p;\beta_1/\alpha_1,\beta_2/\alpha_2)$ in Figure \ref{fig-3b-inf} and that
$L$ admits infinitely many 3-bridge spheres $\{P^i\}_{i\in\Z}$ up to isotopy.
\end{itemize}

\begin{subcase2}\label{case2-4}
$M$ satisfies the condition (M4).
\end{subcase2}

Note that $M$ is obtained by gluing $M_1, M_2\in D[2]$ and $M_3=E(K_1\cup K_2)$, where $K_1\cup K_2$ is a hyperbolic 2-bridge link with components $K_1$ and $K_2$,
so that the regular fiber of $M_i$ is identified with the meridian loop of $K_i$ ($i=1,2$).
Recall from Theorem \ref{thm-kob} that,
for any genus-2 Heegaard surface $F$ of $M$,
$F\cap M_i$ ($i=1,2$) is an essential saturated annulus in $M_1$ 
and $F\cap M_3$ is the 2-bridge sphere of $K_1\cup K_2$.
Let $D_{\mu_i}$ and $D_{\lambda_i}$ $(i=1,2)$ denote the Dehn twists along $T=\partial M_1=\partial M_2$ in the direction of the meridian and a longitude of $K_i$, respectively. 
Note that $K_1\cup K_2$ admits a unique 2-bridge sphere up to isotopy by \cite{Sch},
and that $M_i$ contains a unique essential annulus up to isotopy.
Let $F_0$ be a \lq\lq standard\rq\rq\  genus-2 Heegaard surface of $M$.
Then, by \cite[Lemma 6 (1)]{Jan2},
any genus-2 Heegaard surface $F$ of $M$ is isotopic to 
$D_{\lambda_1}^{n_1/2}D_{\lambda_2}^{n_2/2}(F_0)$
for some integers $n_1$ and $n_2$ by \cite[Lemma 6 (1)]{Jan2}.
Since $\tau_{F_0}D_{\lambda_i}\tau_{F_0}=D_{\lambda_i}^{-1}$,
we have, by \cite[Lemma 5]{Jan2}, 
$$\tau_{D_{\lambda_1}^{n_1/2}D_{\lambda_2}^{n_2/2}(F_0)}
=D_{\lambda_1}^{n_1}D_{\lambda_2}^{n_2}\tau_{F_0}.$$
Since $\D=\langle D_{\lambda_1}, D_{\lambda_2}\rangle\cong\Z^2$ (see Lemma \ref{lem-d} (6)),
we have $D_{\lambda_1}^{n_1}D_{\lambda_2}^{n_2}\tau_{F_0}=\tau_{F_0}$ if and only if $n_1=n_2=0$.
Hence, $M$ admits a unique genus-2 Heegaard surface whose hyper-elliptic involution is $\tau_L$, a contradiction.

This completes the proof of Theorem \ref{thm-main}.

\section{Proof of Proposition \ref{prop-toruslink}}\label{sec-toruslink}

Let $L$ be a torus link $T(3,3n')$ for some nonzero integer $n'$, and
let $K_1$, $K_2$ and $K_3$ be the three components of $L$.
Let $S$ be a 3-bridge sphere of $L$.
Let $T$ be the standard torus in $S^3$ containing $L$,
and let $A_i$ $(i=1,2,3)$ be the closure of a component of $T\setminus L$ 
bounded by two components of $L$ different from $K_i$.
Let $V_1$ and $V_2$ be the two solid tori in $S^3$ bounded by $T$
such that the meridians of $V_1$ and $V_2$ meet $L$ in three points and $|3n'|$ points, respectively.
Since $S\cap K_i$ consists of two points for each $i=1,2,3$,
$S\cap A_i$ satisfies one of the following conditions (see Figure \ref{fig-int}).
\begin{itemize}
\item[(i)] $S\cap A_i$ contains two non-separating arcs $\gamma_i^1$ and $\gamma_i^2$.
\item[(ii)] $S\cap A_i$ contains two separating arcs $\gamma_i^1$ and $\gamma_i^2$.
\end{itemize}
%
\begin{figure}
\begin{center}
\includegraphics*[width=6cm]{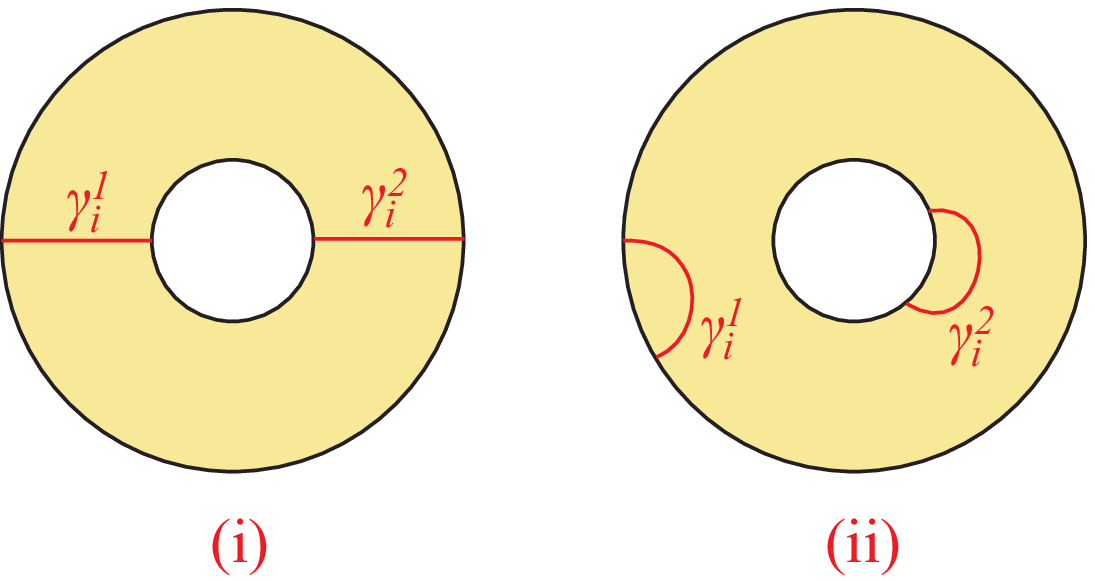}
\end{center}
\caption{}
\label{fig-int}
\end{figure}
%
Thus one of the following holds.
\begin{itemize}
\item[(S1)] $S\cap A_i$ satisfies the condition (i) for every $i=1,2,3$.
\item[(S2)] Two of $S\cap A_1$, $S\cap A_2$ and $S\cap A_3$ satisfies the condition (i) and the other satisfies the condition (ii).
\item[(S3)] Two of $S\cap A_1$, $S\cap A_2$ and $S\cap A_3$ satisfies the condition (ii) and the other satisfies the condition (i).
\item[(S4)] $S\cap A_i$ satisfies the condition (ii) for every $i=1,2,3$.
\end{itemize}

Assume that $S$ satisfies the condition (S1).
Note that $\gamma:=\gamma_1^1\cup\gamma_1^2\cup\gamma_2^1\cup\gamma_2^2\cup\gamma_3^1\cup\gamma_3^2$
consists of two loops each of which contains one of the two points $S\cap K_i$ for every $i=1,2,3$.
Suppose there is a loop component, $\delta$, of $S\cap T$ other than $\gamma$.
Then $\delta$ bounds a disk in the interior of $A_i$ disjoint from $\gamma$ for some $i=1,2,3$,
and hence, $\delta\cup K_i$ is a trivial 2-component link for each $i=1,2,3$.
On the other hand, $\delta$ either bounds a disk in $S\setminus \gamma$ 
or is isotopic to the core of the annulus component of $S\setminus \gamma$.
In the latter case, the linking number of $\delta$ and a component of $L$ is $1$
(see Figure \ref{fig-int2}), a contradiction.
%
\begin{figure}
\begin{center}
\includegraphics*[width=3.3cm]{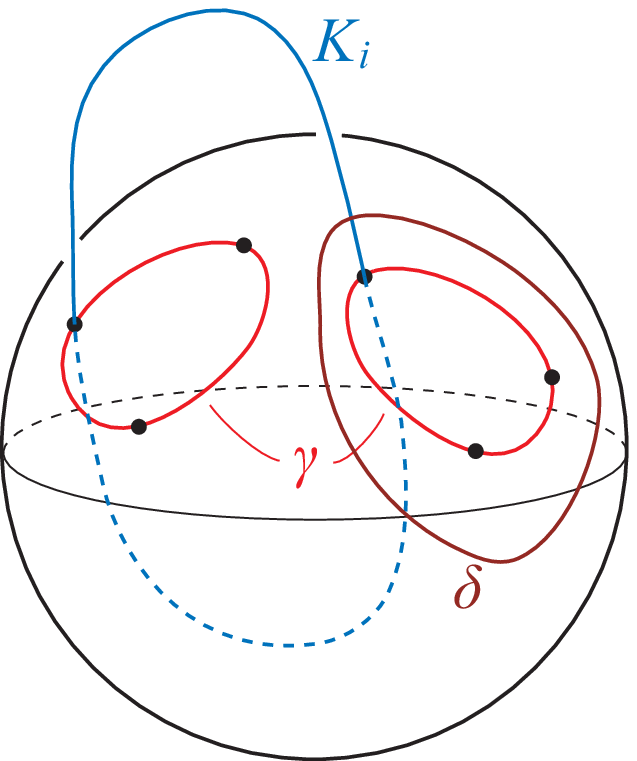}
\end{center}
\caption{}
\label{fig-int2}
\end{figure}
%
Hence, $\delta$ bounds a disk in $S\setminus \gamma$.
If we assume that $\delta$ is innermost (in $S\setminus \gamma$),
then we can eliminate it from $S\cap T$ 
since the union of the two disks described above is a 2-sphere bounding a 3-ball in $S^3\setminus L$.
Hence, we may assume that $S\cap T=\gamma$.
Since $\gamma$ cuts $S$ into two disks and an annulus,
it bounds two disks in $V_1$ and bounds an annulus in $V_2$.
Since such a disk is unique up to isotopy in $(V_1,L)$
and an annulus is unique up to isotopy in $(V_2,L)$, 
$L$ admits a unique 3-bridge sphere satisfying the condition (S1).
(We obtain two 3-bridge spheres when $n'=\pm 1$, but it can be easily seen that they are isotopic.)

Assume that $S$ satisfies the condition (S2).
Note that $\gamma:=\gamma_1^1\cup\gamma_1^2\cup\gamma_2^1\cup\gamma_2^2\cup\gamma_3^1\cup\gamma_3^2$
is a loop containing the six points $S\cap L$.
Thus any loop component of $S\cap A_i$ except $\gamma$ bounds a disk in $S\setminus\gamma\subset S^3\setminus L$.
This implies that 
any loop component of $S\cap A_i$ cannot be a core of the annulus $A_i$ for each $i=1,2,3$,
since the linking number of the core of $A_i$ and a component of $L$ is $n'(\neq 0)$.
Hence, any loop component of $S\cap A_i$ also bounds a disk in the interior of $A_i$ disjoint from $\gamma$,
and hence, we can remove all such components by an isotopy.
Thus we may assume that $S\cap T$ consists of only one loop component $\gamma$.
Since $\gamma$ itself bounds disks in $S$ on both sides, 
it must be an inessential loop on $T$.

We show that this case can be reduced to the case where $S$ satisfies the condition (S1).
To this end, let $h:S^3\rightarrow [-2,2]$ be the height function 
such that 
$S_t:=h^{-1}(t)$ is a 3-bridge sphere of $L$ when $-1<t<1$,
that $S_t$ is a 2-sphere which meets each $K_i$ in one point when $t=\pm 1$,
that $S_t$ is a single point when $t=\pm 2$
and that $S_t$ is a 2-sphere in $S^3\setminus L$ otherwise.
Moreover, we may assume that $S_{0}=S$
and that the restriction $g:=h|_{T}$ of $h$ to $T$ has 
at most one non-degenerate singular point
at every level.
Thus, for every singular value $t_0$, 
$g^{-1}({t_0})$ contains a maximal point, a minimal point or a saddle point.
We represent each saddle point in $g^{-1}({t_0})$ by an arc on $T$
with endpoints on $g^{-1}({t_0-\varepsilon})$ 
for sufficiently small $\varepsilon >0$, as in Figure \ref{fig-saddle}.
%
\begin{figure}
\begin{center}
\includegraphics*[width=6cm]{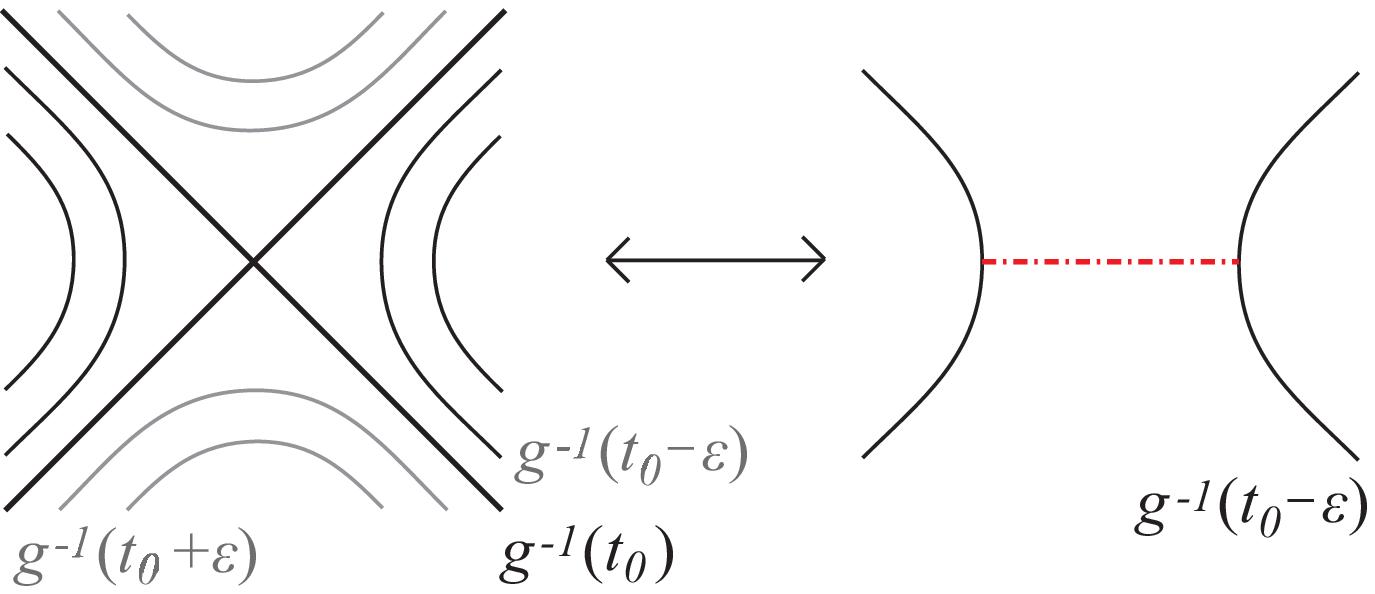}
\end{center}
\caption{}
\label{fig-saddle}
\end{figure}
%
Such an arc, $\alpha$, is of one of the following three types 
(see Figure \ref{fig-type}):
%
\begin{figure}
\begin{center}
\includegraphics*[width=7cm]{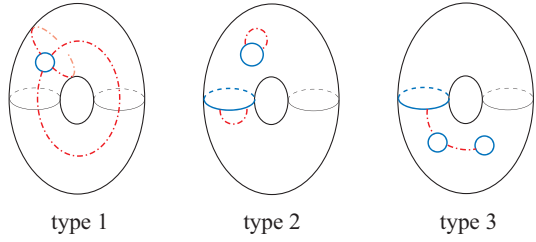}
\end{center}
\caption{The dashed and dotted lines give all possible types of an arc $\alpha$ representing a saddle point of $g$.}
\label{fig-type}
\end{figure}
%
\begin{itemize}
\item
$\alpha$ is of {\it type 1} 
if its endpoints are on the same component of $g^{-1}({t_0-\varepsilon})$, 
and $g^{-1}({t_0+\varepsilon})$ contains a pair of parallel essential loops on $T$,
\item
$\alpha$ is of {\it type 2} 
if its endpoints are on the same component of $g^{-1}({t_0-\varepsilon})$, 
and $g^{-1}({t_0+\varepsilon})$ does not contain a pair of parallel essential loops on $T$,
and 
\item
$\alpha$ is of {\it type 3} 
if its endpoints are on different components of $g^{-1}({t_0-\varepsilon})$.
\end{itemize}

Put $X_{s}:=g^{-1}([-2,s])$ for any $s\in [-2,2]$. 
Since $S(=S_{0})$ cuts $T$ into a disk and a 1-holed torus,
we may assume that $X_{0}$ is the disk.
Since $L\subset X_1$, we see that $X_1$ contains an essential loop on $T$.
Thus, there exists a singular value $s_0>0$ 
and a sufficiently small $\varepsilon>0$
such that
$X_{s_0-\varepsilon}$ does not contain an essential loop on $T$ and
$X_{s_0+\varepsilon}$ contains an essential loop on $T$.
Note that, if $X_{s_0-\varepsilon}$ does not contain an essential loop on $T$
and the arc representing the singular point at $s_0$ is of type 2 or of type 3,
then $X_{s_0+\varepsilon}$ cannot contain an essential loop on $T$.
Thus, the arc representing the singular point at $s_0$ must be of type 1,
and hence, $g^{-1}({s_0+\varepsilon})$ contains a pair of parallel essential loops, say $c$ and $c'$, on $T$.
Note that $c$ bounds a $k$-holed disk in $S':=S_{s_0+\varepsilon}$ disjoint from $g^{-1}({s_0+\varepsilon})\setminus (c\cup(\cup_{i=1}^{k}c_k))$
together with $k$ components $c_1,c_2,\dots,c_k$ of $g^{-1}({s_0+\varepsilon})\setminus (c\cup c')$ for some non-negative integer $k$.
We see that $c$ is null-homologous in $V_i$ for some $i=1,2$
since $c_1,c_2,\dots,c_k$ are inessential loops on $T$.
(By the choice of $s_0$, all components of $g^{-1}(s_0+\varepsilon)$ except $c$ and $c'$ are inessential on $T$.)
Hence, $c$ and $c'$ must be meridian loops of one of the solid tori $V_1$ and $V_2$.
Since each of $c$ and $c'$ meets each component of $L$ in a single point,
it intersects each of the annuli $A_1$, $A_2$ and $A_3$ in a non-separating arc.
Hence, the 3-bridge sphere $S'$, which is isotopic to $S$, satisfies the condition (S1).

Similarly, the cases where $S$ satisfies the condition (S3) or (S4) is reduced to the first case.
This implies that $L$ admits a unique 3-bridge sphere up to isotopy.

This completes the proof of Proposition \ref{prop-toruslink}.

\vspace{2mm}

{\bf Acknowledgment}
I would like to express my appreciation to Professor
Makoto Sakuma for his guidance, advice and encouragement.
I would also like to thank 
Professor Susumu Hirose for his helpful comments, especially on Section \ref{sec-mcg}.


\begin{thebibliography}{9}

\bibitem{Bir3} J. S. Birman,
{\it Braids, links, and mapping class groups},
Annals of Mathematics Studies, No. 82. 
Princeton University Press, Princeton, N.J

\bibitem{Boi} M. Boileau, D. J. Collins and H. Zieschang,
{\it Genus 2 Heegaard decompositions of small Seifert 3-manifolds},
Ann. Inst. Fourier (Grenoble) {\bf 41} (1991), 1005--1024.

\bibitem{Boi7} M. Boileau and J. Porti,
{\it Geometrization of 3-orbifolds of cyclic type},
Ast\'{e}risque {\bf 272} (2001), 208.

\bibitem{Boi3} M. Boileau and H. Zieschang,
{\it Nombre de ponts et g\'{e}n\'{e}rateurs m\'{e}ridiens des entrelacs de Montesinos},
Comment. Math. Helvetici {\bf 60} (1985), 270--279.

\bibitem{Bon}
F. Bonahon and L. Siebenmann,
{\it New geometric splittings of classical knots, and 
the classification and symmetries of arborescent knots},
available at http://www-bcf.usc.edu/~fbonahon.

\bibitem{Bur2}
G. Burde and K. Murasugi,
{\it Links and Seifert fiber spaces}, 
Duke Math. J. {\bf 37} (1970), 89--93.

\bibitem{Coo} D. Cooper, C. Hodgson and S. Kerckhoff,
{\it Three-dimensional orbifolds and cone-manifolds},
MSJ Memoirs, 5. Mathematical Society of Japan, Tokyo (2000). 

\bibitem{Dun}
W. D. Dunbar,
{\it Geometric orbifolds}, 
Rev. Mat. Univ. Complut. Madrid {\bf 1} (1988), no. 1--3, 67--99.

\bibitem{Jan} Y. Jang,
{\it Three-bridge links with infinitely many three-bridge spheres},
Topology Appl. {\bf 157} (2010), 165--172.

\bibitem{Jan2} Y. Jang,
{\it Classification of 3-bridge arborescent links},
Hiroshima Math. J. {\bf 41} (2011), 89--136.

\bibitem{Jan3} Y. Jang,
{\it Classification of 3-bridge spheres of 3-bridge arborescent links},
preprint, available at http://www.geocities.jp/yyyjang.

\bibitem{Joh} K. Johannson,
{\it Homotopy equivalences of $3$-manifolds with boundaries}, 
Lecture Notes in Mathematics, 761. Springer, Berlin, 1979. ii+303 pp.

\bibitem{Joh2} K. Johannson,
{\it Heegaard surfaces in Haken $3$-manifolds}, 
Bull. Amer. Math. Soc. (N.S.) {\bf 23} (1990), no.1, 91--98.

\bibitem{Joh3} D. L. Johnson,
{\it Presentation of groups}, 
London Math. Soc. Student Texts, 15. Cambridge University Press,
Cambridge, 1990.

\bibitem{Kob} T. Kobayashi,
{\it Structures of the Haken manifolds with Heegaard splittings of genus two},
Osaka J. Math. {\bf 21} (1984), 437--455.

\bibitem{Kob2} T. Kobayashi,
{\it Nonseparating incompressible tori in $3$-manifolds},
J. Math. Soc. Japan {\bf 36} (1984), no.1, 11--22.

\bibitem{Li} T. Li,
{\it Heegaard surfaces and measured laminations. I. The Waldhausen conjecture},
Invent. Math. {\bf 167} (2007), no.1, 135--177.

\bibitem{Men} W.~Menasco,
Closed incompressible surfaces in alternating knot and link complements,
Topology, {\bf 23} (1984), 37--44.

\bibitem{Mor} K.~Morimoto,
{\it On minimum genus Heegaard splittings of some orientable closed 3-manifolds},
Tokyo J. Math. {\bf 12} (1989), 321--355.

\bibitem{Ota1} J.-P. Otal,
{\it Pr\'{e}sentations en ponts du n\oe ud trivial},
C. R. Acad. Sci. Paris S\'{e}r. I Math. {\bf 294} (1982), no. 16, 553--556.

\bibitem{Ota2} J.-P. Otal,
{\it Pr\'{e}sentations en ponts des n\oe ud rationnels},
Low-dimensional topology (Chelwood Gate, 1982), 143--160,
London Math. Soc. Lecture Note Ser., 95, Cambridge Univ. Press, Cambridge, 1985.

\bibitem{Sak3} M. Sakuma,
{\it Realization of the symmetry groups of links}, 
Transformation groups (Osaka, 1987), 291--306, 
Lecture Notes in Math., 1375, Springer, Berlin, 1989.

\bibitem{Sch2} M. Scharlemann and M. Tomova,
{\it Uniqueness of bridge surfaces for 2-bridge knots},
Math. Proc. Cambridge Philos. Soc. {\bf 144} (2008), no. 3, 639--650.

\bibitem{Sch} H. Schubert,
{\it Knoten mit zwei Br\"{u}cken},
Math. Z. {\bf 65} (1956), 133--170.

\end{thebibliography}
\end{document}